\documentclass{amsart}

\usepackage[utf8]{inputenc}
\usepackage[english]{babel}

\usepackage{amsfonts,amssymb,amsmath}

\usepackage{xcolor}
\usepackage[colorlinks,
    linkcolor={red!50!black},
    citecolor={blue!50!black},
    urlcolor={blue!80!black}]{hyperref}

\usepackage{tikz}
\usepackage[all]{xy}
\usetikzlibrary{calc, matrix, arrows, cd}

\newtheorem{theorem}{Theorem}[section]

\newtheorem{corollary}[theorem]{Corollary}
\newtheorem{lemma}[theorem]{Lemma}
\newtheorem{proposition}[theorem]{Proposition}

\theoremstyle{definition}
\newtheorem{definition}[theorem]{Definition}

\newtheorem{example}[theorem]{Example}
\newtheorem{construction}[theorem]{Construction}

\theoremstyle{remark}
\newtheorem*{claim}{Claim}
\newtheorem*{claim*}{Claim}
\newtheorem{remark}[theorem]{Remark}

\newcommand{\Z}{\mathbb{Z}}
\newcommand{\Q}{\mathbb{Q}}
\newcommand{\R}{\mathbb{R}}
\newcommand{\C}{\mathbb{C}}

\newcommand{\T}{\mathbb{T}}
\newcommand{\F}{\mathbb{F}}
\newcommand{\id}{\operatorname{id}}
\newcommand{\sign}{\operatorname{sign}}
\newcommand{\eps}{\varepsilon}
\newcommand{\im}{\operatorname{im}}

\newcommand{\Cyl}{\operatorname{Cyl}}
\newcommand{\PD}{\operatorname{PD}}
\newcommand{\Pb}{\operatorname{Pb}}
\newcommand{\sm}{\setminus}
\newcommand{\aug}{\operatorname{aug}}

\newcommand{\tr}{\text{tr}}

\newcommand{\op}{\operatorname}

\newcommand{\dsign}{\operatorname{dsign}}
\DeclareMathOperator{\lk}{lk}

\DeclareMathOperator{\coker}{coker}

\DeclareMathOperator{\Hom}{Hom}

\DeclareMathOperator{\rk}{rk}

\begin{document}

\title{Multivariable signatures, genus bounds and \texorpdfstring{$0.5$--solvable}{0.5-solvable} cobordisms}
\author{Anthony Conway}
\address{Universit\'e de Gen\`eve, Section de math\'ematiques, 2-4 rue du Li\`evre, 1211 Gen\`eve 4, Switzerland}
\email{anthony.conway@unige.ch}

\author{Matthias Nagel}
\address{McMaster University, Hamilton, Canada}
\email{nagel@cirget.ca}

\author{Enrico Toffoli}
\address{Fakult\"at f\"ur Mathematik\\ Universit\"at Regensburg\\   Germany}
\email{enricotoffoli@gmail.com}

\begin{abstract}
We refine prior bounds on how the multivariable signature and the nullity of a link change
under link cobordisms. The formula generalizes a series of results about the 4-genus having
their origins in the Murasugi-Tristram inequality, and at the same time extends previously
known results about concordance invariance of the signature to a bigger set of allowed variables.
Finally, we show that the multivariable signature and
nullity are also invariant under $0.5$--solvable cobordism.
\end{abstract}
\maketitle

\section{Introduction}\label{sec:Intro}

Given $\omega \in S^1\sm\{1\}$, the Levine-Tristram signature and nullity of a link~$L$ are given
by the signature and nullity of $(1-\omega)A+(1-\overline{\omega})A^T$, where
$A$ is any Seifert matrix for $L$~\cite{Levine69, Tristram}.  For
a~$\mu$-colored link, i.e.\ an oriented link~$L$ in~$S^3$ whose components are
partitioned into~$\mu$ sublinks~$L_1, \dotsc , L_\mu$, the
Levine-Tristram signature and nullity have been generalized to multivariable functions
\[ \sigma_L, \eta_L \colon \T^\mu \to \Z, \]
where $ \T^\mu$ denotes the set $(S^1 \setminus \{1\})^\mu$~\cite{CimasoniFlorens}.
Apart from their $3$-dimensional definition using C-complexes~\cite{Cooper, CimasoniFlorens},
a $4$-dimensional interpretation in the smooth setting has been given by
Cimasoni-Florens using branched covers and the $G$-signature theorem
for elements of $\T^\mu$ of finite order~\cite[Theorem 6.1]{CimasoniFlorens}.
We focus on another interpretation by Viro~\cite{Viro09} using directly the complements of surfaces bounding the link in the $4$-ball.

We shall always work in the topological (locally flat) category.
Let $F$ be a union~$F_1 \cup \dots \cup F_\mu \subset D^4$ of properly embedded locally flat surfaces
that only intersect each other transversally in double points and whose boundary is
a colored link~$L \subset S^3$.
Since the first homology group of the exterior~$W_F$ of  such a \emph{colored bounding surface} $F \subset D^4$ is free abelian,
any choice of $\omega \in \T^\mu$ gives rise to a coefficient system $H_1(W_F;\Z) \to U(1)$
and thus to a twisted signature~$\sign_\omega(W_F)$.
The twisted signature~$\sign_\omega (W_F)$ is independent of the
colored bounding surface $F$
and defines an invariant of colored links~\cite[Section 2.3]{Viro09}.
Building on \cite[Theorem 1.3]{ConwayFriedlToffoli}, we give a proof to the following statement of~\cite[Section 2.5]{Viro09} in Proposition~\ref{prop:ColorSignature}. The corresponding result for the nullity is proven in Proposition~\ref{prop:NullityNoCcomplex}.

\begin{proposition}\label{prop:ViroCimasoniFlorens}
Let $L$ be a $\mu$-colored link and let $\omega \in \T^\mu$. For any colored bounding surface $F$, the twisted signature $\sign_\omega(W_F)$ coincides with the multivariable signature $\sigma_L(\omega)$.
\end{proposition}

Cimasoni and Florens showed that the signature $\sigma_L(\omega)$ is invariant under smooth link concordance~\cite[Theorem 7.1]{CimasoniFlorens} for those $\omega=(\omega_1, \dots, \omega_\mu) \in \T^\mu$
that satisfy the following condition: there exists a prime $p$ such that for all $i$, the order of $\omega_i$ is a power of $p$. For the same subset of $\T^\mu$, they provide lower bounds on the genus and on the number of double points of smooth surfaces in $D^4$ bounded by a colored link $L$~\cite[Theorem 7.2]{CimasoniFlorens}, extending the
Murasugi-Tristram inequality~\cite{Murasugi,Tristram}  to the multivariable setting.

Building on the approach used in~\cite{NagelPowell} to study
concordance invariance of the Levine-Tristram signature, we consider the subset $\T_!^\mu$ of $\T^\mu$ given by those $\omega$'s which are not roots of any polynomial  $p\in  \Z[t_1^{\pm 1},\dots ,t_\mu^{\pm 1}]$ whose evaluation on $(1,\dotsc, 1)$ is invertible. This set includes the elements considered by Cimasoni and Florens~\cite[Section 7]{CimasoniFlorens}; see Proposition~\ref{prop:TPContainedT!}.
A \emph{colored cobordism} between two $\mu$-colored links $L$ and $L'$ is a collection of  properly embedded locally flat surfaces~$\Sigma = \Sigma_1 \cup \dots \cup \Sigma_\mu$
in $S^3 \times [0,1]$ which have the following properties: the surfaces only intersect each other transversally in double points, each surface~$\Sigma_i$ has boundary~$L_i \sqcup -L_i'$, and each connected component of $\Sigma_i$ has at least one boundary component in $S^3 \times \{0\}$
and one in $S^3 \times \{1\}$.
Our first main result gives bounds on the Euler characteristic and on the number of double points in such a cobordism, generalizing Powell's treatment of a genus bound for the Levine-Tristram signature~\cite{Powell}.

\begin{theorem}\label{thm:GenusIntro}
Let $\Sigma = \Sigma_1 \cup \dots \cup \Sigma_\mu$ be a colored cobordism between two $\mu$-colored links $L$ and $L'$. If $\Sigma$ has $c$ double points, then
\[|\sigma_L(\omega)-\sigma_{L'}(\omega)|  + |\eta_L(\omega)-\eta_{L'}(\omega)|
	\leq \sum_{i=1}^{\mu} -\chi(\Sigma_i) + c \]
for all $\omega\in \T_!^\mu$.
\end{theorem}

Two $\mu$-colored links~$L$ and~$L'$ are \emph{concordant} if
there exists a $\mu$-colored cobordism between $L$ and $L'$ that has no
intersection points and consists exclusively of annuli.
As an application of Theorem~\ref{thm:GenusIntro}, we extend two different results of Cimasoni and Florens to the topological setting and to a bigger set of values of the variable $\omega$. The first result relaxes the conditions under which the signature and nullity are an obstruction to colored concordance~\cite[Theorem $7.1$]{CimasoniFlorens}. See Corollary~\ref{cor:ConcordanceViaGenus} for a proof.

\begin{corollary}
\label{cor:ConcordanceIntro}
The multivariable signature and nullity are topological concordance invariants at all $\omega \in \mathbb{T}^\mu_!$
\end{corollary}

As a second consequence of Theorem~\ref{thm:GenusIntro}, we obtain a generalization of~\cite[Theorem $7.2$]{CimasoniFlorens}; the latter result being itself an extension of the Murasugi-Tristram inequality~\cite{Murasugi, Tristram}. In what follows, we denote the first Betti number of a surface~$F$ by $\beta_1(F)$. We refer the reader to
Corollary~\ref{cor:CimasoniFlorens72} for a proof of the next result and to Remark~\ref{rem:Genus}
for a comparison with a similar result obtained by Viro~\cite[Section $4$]{Viro09}.

\begin{corollary}\label{cor:CimasoniFlorens72Intro}
Let $F=F_1 \cup \cdots \cup F_\mu$ be a colored bounding surface for a $\mu$-colored link $L$
such that $F_1,\dots, F_\mu$ have a total number of $m$ connected components,
intersecting in $c$ double points. 
Then, for all $\omega \in \T_!^\mu$, we have
\[  |\sigma_L(\omega)|+|\eta_L(\omega)-m+1| \leq \sum_{i=1}^\mu \beta_1(F_i) +c. \]
\end{corollary}

The last part of this article deals with $0.5$-solvable cobordisms.
This notion was defined by Cha~\cite{Cha} giving a relative version of the notion of Cochran-Orr-Teichner's $n$-solvability~\cite{CochranOrrTeichner}.
We refer to Section~\ref{sec:Solvable} for the precise definition of $n$-solvable cobordant links, however note that abelian link invariants are not expected to distinguish $0.5$-solvable cobordant links.
For instance, if two links are $1$-solvable cobordant,
then their first non-zero Alexander polynomials agree up to norms and their Blanchfield pairings are Witt
equivalent~\cite[Theorems $B$ and $C$]{Kim}. Our final result is the corresponding
statement for the multivariable signature and nullity.
\begin{theorem}\label{thm:SolvableNullitySignature}
If two $\mu$-colored links $L$ and $L'$ are $0.5$-solvable cobordant, then
\[\eta_{L}(\omega)=\eta_{L'}(\omega) \quad \text{and} \quad \sigma_{L}(\omega)=\sigma_{L'}(\omega)\]
for all $\omega \in \mathbb{T}^m_!$.
\end{theorem}
Since concordant links are $n$-solvable cobordant for all $n$, Theorem~\ref{thm:SolvableNullitySignature} can be viewed as a vast refinement of Corollary~\ref{cor:ConcordanceIntro}.

\begin{remark}\label{rem:WhitneyGropeFormulation}
Note that the notion of $n$-solvable cobordism is related to Whitney tower/grope concordance.  See \cite{Cha}
for the definition of these notions. In particular, using~\cite[Corollary 2.17]{Cha},
Theorem~\ref{thm:SolvableNullitySignature} implies that the multivariable signature and
nullity are invariant under height~$3$ Whitney tower/grope concordance.
\end{remark}

\begin{remark}
The \emph{Alexander nullity} $\beta(L)$ of a colored link $L$ is the $\Z[t_1^{\pm 1},\ldots,t_\mu^{\pm 1}]$-rank of its Alexander module. Kim~\cite[Theorem C]{Kim} showed that the Alexander nullity is invariant under $1$-solvable cobordisms. In Proposition~\ref{prop:Invariance}, we improve this result by proving invariance under $0.5$-solvable cobordisms. Note also that this statement does not follow from the invariance of the nullity function $\eta_L(\omega)$ since $\beta(L)=\operatorname{min}\lbrace \eta_L(\omega) \ | \ \omega \in \T^\mu \rbrace$~\cite[Proposition 2.3]{CimasoniConwayZacharova}.
\end{remark}

\medbreak
This paper is organized as follows. Section~\ref{sec:Prelim} introduces the necessary
background material on twisted homology and signatures.
Section~\ref{sec:4dDef} introduces the colored signature and nullity and
proves Theorem~\ref{thm:GenusIntro} together with its applications.
Section~\ref{sec:Plumbed} introduces plumbed $3$-manifolds and proves some results
about their signature defects. These form the technical foundation for the proof of Theorem~\ref{thm:SolvableNullitySignature}, which is the subject of Section~\ref{sec:Solvable}.

\subsection*{Acknowledgments.}
We thank Christopher Davis for sharing his insights of $0.5$--solvability,
which helped us immensely in navigating through the technicalities of Section~\ref{sec:Solvable}.
The authors wish to thank Ana Lecuona, David Cimasoni, Vincent Florens, Stefan Friedl, Paul Kirk, Andrew Nicas and Mark Powell for helpful discussions.
We are indebted to the referees for their detailed and helpful suggestions.
AC thanks UQ\`AM for its hospitality and was supported by the NCCR SwissMap funded by the Swiss FNS.
MN is grateful for his stay at the University of Regensburg funded by the SFB 1085, which started the project.
ET was supported by the GK ``Curvature, Cycles and Cohomology'',
funded by the Deutsche Forschungsgemeinschaft (DFG).
MN was supported by a CIRGET postdoctoral fellowship, and by a Britton postdoctoral fellowship from McMaster University.

\section{Twisted homology, signatures and concordance roots}\label{sec:Prelim}
In Section~\ref{sub:Twisted}, we set up the conventions on twisted homology.
In Section~\ref{sub:IntersectionForm},
we review twisted intersection forms, which leads us to discuss the additivity of the signature in Section~\ref{sub:NovikovWall}.
In Section~\ref{sub:ConcordanceRoots}, we generalize the concept of Knotennullstellen~\cite{NagelPowell}.

\subsection{Twisted homology}\label{sub:Twisted}
We start by fixing some notation and conventions regarding twisted homology. After that, we review two universal coefficient spectral sequences and apply them to a particular abelian coefficient system.
\medbreak
Let $X$ be a connected CW-complex and let $Y \subset X$ be a possibly empty subcomplex.
Denote by $p \colon \widetilde{X} \to X$ the universal cover of $X$ and set
$\widetilde{Y}:=p^{-1}(Y)$, so that $C(\widetilde{X},\widetilde{Y})$ is a
left $\mathbb{Z}[\pi_1(X)]$-module. Given a ring $\F$ with involution, we can consider
homomorphisms~$\phi \colon \mathbb{Z}[\pi_1(X)] \to \F$ of rings with involutions, which
means that $\phi(g^{-1}) = \overline {\phi(g)}$ for all $g \in \pi_1(X)$. Such a homomorphism~$\phi$
turns $\F$ into a $(\F,\mathbb{Z}[\pi_1(X)])$-bimodule, which we denote by~$R$. We may consider the left $\F$--modules
\begin{align*}
H_*(X,Y;R)&=H_* \left(R \otimes_{\mathbb{Z}[\pi_1(X)]} C(\widetilde{X},\widetilde{Y}) \right), \\
H^*(X,Y;R)&=H_*\left( \text{Hom}_{\text{right-}\mathbb{Z}[\pi_1(X)]}(C(\widetilde{X},\widetilde{Y})^\tr,R) \right),
\end{align*}
where the \emph{transposed module}~$M^\tr$ of an $S$-module~$M$ has the same underlying abelian
group with multiplication flipped using the involution.

Our main examples of twisted homology and cohomology modules will come from the following examples.
\begin{example}\label{ex:Ccoefficients}
Let $\varphi \colon \pi_1(X) \to \mathbb{Z}^\mu=\langle t_1,\dots,t_\mu
\rangle$ be a homomorphism and let~$\omega = (\omega_1,\dots,\omega_\mu) \in \mathbb{T}^\mu \subset \mathbb{C}^\mu$.
Composing the induced map
$\mathbb{Z}[\pi_1(X)] \to \mathbb{Z}[\mathbb{Z}^\mu]  $ with the
map~$\mathbb{Z}[\mathbb{Z}^\mu] \xrightarrow{\alpha} \C$ which evaluates $t_i$ at $\omega_i$, produces a morphism~$\phi \colon \mathbb{Z}[\pi_1(X)] \to \mathbb{C}$ of rings with involutions.
In turn, $\phi$ endows $\C$ with a $(\mathbb{C},\mathbb{Z}[\pi_1(X)])$-bimodule structure.
To emphasize the choice of $\omega$, we shall write $\mathbb{C}^\omega$ instead of $\mathbb{C}$.
Since $\mathbb{C}^\omega$ is a $(\mathbb{C},\mathbb{Z}[\pi_1(X)])$-bimodule, we may
consider the complex vector spaces $H_k(X,Y;\mathbb{C}^\omega)$ and
$H^k(X,Y;\mathbb{C}^\omega)$.

Consider the ring $\Lambda_S=\mathbb{Z}[t_1^{\pm 1},\dots,t_\mu^{\pm 1},(1-t_1)^{-1},\dots,(1-t_\mu)^{-1}]$
and observe that since none of the $\omega_i$ are equal to $1$, the map $\phi\colon \Z[\pi_1(X)] \to \mathbb{C}$ factors through a
map $\Lambda_S \to \mathbb{C}$. In particular, the homology
$\mathbb{C}$-vector space~$H_k(X,Y;\C^\omega)$ is the $k$--th homology of the chain complex~$\mathbb{C} \otimes_{\Lambda_S} C(X,Y;\Lambda_S)$.
\end{example}

\begin{example}
\label{ex:FieldCoefficients}
Let $\Q(\Z^\mu)$ denote the field of fractions of $\Lambda:=\Z[t_1^{\pm 1},\ldots,t_\mu^{\pm 1}]$. Given a homomorphism $\varphi \colon \pi_1(X) \to \mathbb{Z}^\mu=\langle t_1,\dots,t_\mu \rangle$, the canonical map $\Lambda \to \Q(\Z^\mu)$ endows $\Q(\Z^\mu)$ with a $(\Q(\Z^\mu),\Z[\pi_1(X)])$--bimodule structure. In particular, we may consider the $\Q(\Z^\mu)$-vector spaces $H_k(X,Y;\Q(\Z^\mu))$. Observe that since $\Q(\Z^\mu)$ is the field of fractions of both $\Lambda$ and $\Lambda_S$, we deduce that $H_k(X,Y;\Q(\Z^\mu))$ is canonically isomorphic to both $\Q(\Z^\mu) \otimes_{\Lambda} H_k(X,Y;\Lambda)$ and $\Q(\Z^\mu) \otimes_{\Lambda_S} H_k(X,Y;\Lambda_S)$.
\end{example}

Most of our main results will involve either the coefficient system $R=\C^\omega$ or the coefficient system $R=\Q(\Z^\mu)$. When we mention that a statement holds for both coefficients systems, it will always be understood that when $R=\C^\omega$ (resp.~$R=\Q(\Z^\mu)$) we take $\F=\C$ (resp. $\F=\Q(\Z^\mu)$).

In order to discuss the relation between homology and cohomology, we introduce
some further notation. First, using the fact that $\phi$ is a morphism of rings
with involution, one can check that
\begin{align*}
 \text{Hom}_{\text{right-}\mathbb{Z}[\pi]}(C(\widetilde{X},\widetilde{Y})^\tr,R) & \to \text{Hom}_{\text{left-}\F}(R \otimes_{\mathbb{Z}[\pi]}C(\widetilde{X},\widetilde{Y}),\F)^\tr   \\
 f & \mapsto \left( (r \otimes \sigma) \mapsto r \overline{f(\sigma)} \right)
\end{align*}
is a well-defined isomorphism of chain complexes of left $\F$-modules.
The isomorphism of chain complexes induces an evaluation homomorphism
\[ \text{ev} \colon H^k(X,Y;R) \to \text{Hom}_{\text{left-}\F}(H_k(X,Y;R),\F)^\tr\]
of left $\F$--modules. This evaluation map is not an isomorphism in general.
Nevertheless, it can be studied using the universal coefficient spectral
sequence~\cite[Theorem 2.3]{Levine77}. For the sake of concreteness, instead of
giving the most general statement, we shall focus on the cases described in Examples~\ref{ex:Ccoefficients} and~\ref{ex:FieldCoefficients}.

\begin{proposition}\label{prop:UCSS}
Let $(X,Y)$ be a CW pair and let $\omega \in \mathbb{T}^\mu$. Suppose $R$ is either~$\C^\omega$ or $\Q(\Z^\mu)$, viewed as a $(\F,\Z[\pi_1(X)])$--module. Then, for each $k$, evaluation provides the following isomorphism of left $\F$-vector spaces:
$$ H^k(X,Y;R) \cong \operatorname{Hom}_{\emph{left-}\F}(H_k(X,Y;R),\F)^\tr.$$
\end{proposition}
\begin{proof}
There is a spectral sequence with $E_2^{p,q}\cong \op{Ext}_\F^q(H_p(X,Y;R),\F)$, which converges to $H^*(X,Y;R)$~\cite[Theorem 2.3]{Levine77}. The result now follows: since $\F$ is a field, the $\op{Ext}$ groups vanish for $q>0$. We also refer to~\cite[Theorem 5.4.4 and Proposition 7.5.4]{ConwayThesis} for further details.
\end{proof}

Given a pair $(X,Y)$, we denote  the rank of $H_i(X,Y)$ by $\beta_i(X,Y)$ and the dimension of $H_i(X,Y;R)$  by $\beta_i^{R}(X,Y)$ when $R$ is either $\C^\omega$ or $\Q(\Z^\mu)$.
As an application of Proposition~\ref{prop:UCSS}, we prove the following lemma.

\begin{lemma}\label{lem:DualityUCSS}
Let $\omega \in \T^\mu$, let $R$ be either $\C^\omega$ or $\Q(\Z^\mu)$ and let $W$ be a $4$-dimensional manifold whose boundary decomposes as $\partial W=M \cup_{\partial} M'$,
where $M$ and $M'$ are (possibly empty) connected $3$-manifolds with $\partial M = \partial M'$.
If $W$ is equipped with a homomorphism~$H_1(W;\Z) \to \Z^\mu$,
then $\beta_{4-i}^{R}(W,M)=\beta_i^{R}(W,M')$ for $i=0,1$.
\end{lemma}
\begin{proof}
By duality, $H_{4-i}(W,M;R) \cong H^i(W,M';R)$. Using Proposition~\ref{prop:UCSS}, we deduce that
$H^i(W,M';R) \cong \text{Hom}_{\F}(H_i(W,M';R),\F)^\tr$ for $i=0,1$. The result now follows immediately.
\end{proof}

As observed in Example~\ref{ex:FieldCoefficients}, there is a canonical isomorphism of $H_k(X,Y;\Q(\Z^\mu))$ with $\Q(\Z^\mu) \otimes_{\Lambda_S} H_k(X,Y;\Lambda_S)$.
On the other hand, a particular case of the universal coefficient spectral sequence in homology is needed to deal with $\C^\omega$-coefficients; see e.g.~\cite[Chapter 2]{Hillman}.
\begin{proposition} \label{prop:UCSSTor}
Given a CW-pair $(X,Y)$ and $\omega \in \mathbb{T}^\mu$, there exists a spectral sequence
\begin{enumerate}
\item converging to $H_*(X,Y;\mathbb{C}^\omega)$
\item with $E^2_{p,q}\cong \text{Tor}^{\Lambda_S}_{p}(H_{q}(X,Y;\Lambda_S),\mathbb{C}^\omega)$
\item with differentials $d^r$ of degree $(-r,r-1).$
\end{enumerate}
More specifically, there is a filtration
$$ 0 \subset F_n^0 \subset F_n^1 \subset \dots \subset  F_n^n=H_n(X,Y;\mathbb{C}^\omega)$$
with $F_n^p/F_n^{p-1} \cong E_{p,n-p}^\infty$.
\end{proposition}

As for cohomology, we provide an easy application of this spectral sequence, to which we shall often refer.
\begin{lemma} \label{lem:UCSSTorExample}
Let $X$ be a connected CW-complex together with a homomorphism $H_1(X;\Z) \to
\mathbb{Z}^\mu = \Z\langle e_1, \dots, e_\mu\rangle$ such that at least one
generator~$e_i$ is in the image.
If $\omega \in \T^\mu$, then
$H_0(X;\mathbb{C}^\omega) = 0$.
Furthermore, $H_1(X;\mathbb{C}^\omega)$ is isomorphic to $\mathbb{C}^\omega \otimes_{\Lambda_S} H_1(X;\Lambda_S)$.
\end{lemma}
\begin{proof}
Using the assumption on the map $H_1(W;\Z) \to \Z^\mu$,
the $\Lambda_S$-module~$H_0(W;\Lambda_S)$ vanishes; see e.g.~\cite[Lemma 2.2]{ConwayFriedlToffoli}).
Thus Proposition~\ref{prop:UCSSTor} immediately implies that $H_0(X;\mathbb{C}^\omega) = 0$. Next, we prove the statement involving $H_1(X;\mathbb{C}^\omega)$. Using the notations of Proposition~\ref{prop:UCSSTor}, the differential $0=\text{Tor}^{\Lambda_S}_2(H_0(X;\Lambda_S),\C^\omega)=E_{2,0} \to E_{0,1}$ is zero. Consequently,~$E_{1,0}^\infty=E_{1,0}^2=0$ and $E_{0,1}^\infty=E_{0,1}^2=\mathbb{C}^\omega \otimes_{\Lambda_S} H_1(X_L;\Lambda_S)$. It follows that $H_1(X_L;\C^\omega)=\mathbb{C}^\omega \otimes_{\Lambda_S} H_1(X_L;\Lambda_S)$, as desired.
\end{proof}

\subsection{Twisted intersection forms and signatures}\label{sub:IntersectionForm}
Here, we review twisted intersection forms. Our main example lies in
the coefficient system introduced in Example~\ref{ex:Ccoefficients}. We
conclude with a short bordism argument showing the vanishing of some signature defects.
\medbreak
Given a compact oriented $n$--dimensional manifold~$W$ and a map $\Z[\pi_1(W)]\to \F$ between rings with involutions. Again, we distinguish the ring~$\F$ from the~$(\F,\Z[\pi_1(W)])$--bimodule~$R$. We denote the Poincar\'e duality isomorphisms by $\PD \colon H_k(W,\partial W;R) \cong H^{n-k}(W;R)$ and~$\PD \colon H_k(W;R) \cong H^{n-k}(W,\partial W;R)$. Composing the map induced by the inclusion $(W,\emptyset) \to (W,\partial W)$ with duality and evaluation produces the map
\[ \Phi \colon H_k(W;R) \to H_k(W,\partial W;R) \xrightarrow{\PD} H^k(W;R) \xrightarrow{\text{ev}}\text{Hom}_{\text{left-}\F}(H_k(W;R),\F)^\tr .\]
The main definition of this section is the following.
\begin{definition}\label{def:int}
The \emph{$R$-twisted intersection pairing}
\[ \lambda_R \colon H_i(W;R) \times H_i(W;R) \to \F \]
is defined by $\lambda_{R}(x,y)=\Phi(y)(x)$.
\end{definition}
The form $\lambda_R$ is hermitian, but need not be
nonsingular. In particular, the space $\im (H_1(\partial W;R) \to H_1(W;R)) $ is annihilated by $\lambda_R$. We conclude this section by giving a crucial example of this set-up.

\begin{example} \label{ex:DerivedSeries}
Let $W$ be a compact connected oriented $4$-manifold. Set $\pi=\pi_1(W)$ and
let $\pi^{(n)}=[\pi^{(n-1)},\pi^{(n-1})]$ denote its derived series starting at $\pi^{(0)}=\pi$.
The projection $\pi \to \pi/\pi^{(n)}$ gives rise to the
$\Z[\pi/\pi^{(n)}]$-modules $H_k(W;\Z[\pi/\pi^{(n)}])$  and we may consider the
$\Z[\pi/\pi^{(n)}]$-twisted intersection pairing
\[ \lambda_n \colon H_2(W;\mathbb{Z}[\pi/\pi^{(n)}]) \times H_2(W;\mathbb{Z}[\pi/\pi^{(n)}] )
	\to \mathbb{Z}[\pi/\pi^{(n)}],\]
as in Definition~\ref{def:int}.  Of particular interest to us is the case where~$n=1$ and $\pi/\pi^{(1)}=H_1(W; \Z)$ is free abelian of rank $\mu$.
In this case, $\Z[\pi/\pi^{(1)}]$ is nothing but the commutative ring~$\Lambda=\Z[t_1^{\pm 1},\dots,t_\mu^{\pm 1}]$ of Laurent polynomials.
\end{example}

We now consider the twisted intersection form in the setting of Example~\ref{ex:Ccoefficients}.
Let $W$ be a $4$-dimensional manifold with (possibly empty) boundary together with a
map~$\varphi \colon \pi_1(W) \to \mathbb{Z}^\mu=\Z\langle t_1,\dots,t_\mu \rangle$.
Given an element~$\omega \in \mathbb{T}^\mu \subset \mathbb{C}^\mu$,
we equip the ring~$\mathbb{C}$ with the $(\C,\Z[\pi_1(W)])$-module structure
described in Example~\ref{ex:Ccoefficients} and consider the
$\mathbb{C}$-vector spaces~$H_k(W;\mathbb{C}^\omega)$.
As in Definition~\ref{def:int}, we may consider the twisted intersection form
\[ \lambda_{\mathbb{C}^\omega} \colon H_2(W;\mathbb{C}^\omega) \times H_2(W;\mathbb{C}^\omega)  \to\mathbb{C}.\]
We write $\sign_\omega W = \sign \lambda_{\C^\omega}$ and $\sign W$ for the untwisted
signature~$\sign \lambda_{\Q}$. We will usually be interested in the \emph{signature defect}
\[ \dsign_\omega W := \sign_\omega W-\sign W.\]

\begin{remark}\label{rem:SignatureRemark}
For a smooth closed manifold of even dimension,
the twisted signature coincides with the untwisted one
and hence the signature defect vanishes.
This can be seen by considering the twisted and untwisted Hirzebruch signature
formula~\cite[Theorem 4.7]{Berline92}, which agree if the bundle carries a flat
connection.
\end{remark}
We prove the corresponding result for topological closed $4$-manifolds over $\Z^\mu$
and give a proof, which does not use index theory.
\begin{proposition}\label{prop:bordism}
Let $Z$ be an oriented $4$-manifold  with a map $\pi_1(Z) \to \mathbb{Z}^\mu$.
If $Z$ is closed, then
$\dsign_\omega Z =0$ for all $\omega \in \mathbb{T}^\mu$.
\end{proposition}
\begin{proof}
Given a space $X$, recall that the bordism group $\Omega_n(X)$ consists of
bordism classes of pairs $(N,\psi)$, where $N$ is an $n$-dimensional manifold
and $\psi \colon N \to X$ is a map; see~\cite{ConnerFloyd} for details.
Moreover, if $G$ is a group with classifying space $BG$, then $\Omega_n(G)$ is
defined as $\Omega_n(BG)$.  Since the choice of the map $\varphi \colon
\pi_1(Z) \to \mathbb{Z}^\mu$ is equivalent to the choice of a homotopy class of
a map~$Z \to T^\mu = B\Z^\mu$, the pair $(Z,\varphi)$ produces an element in
$\Omega_4(\mathbb{Z}^\mu)$.  As both the ordinary and the twisted signature vanish on closed oriented $4$-manifold which bound over $\Z^\mu$, for every $\omega\in \T^\mu$, the signature defect gives rise to a well-defined homomorphism
\[\dsign_\omega \Omega_4(\Z^\mu)\to \Z. \]
We want to prove that $\dsign_\omega$ is the trivial homomorphism.

By the Atiyah-Hirzebruch spectral
sequence~\cite[Chapter $1$, Section $7$]{ConnerFloyd}, we have an
isomorphism
\begin{align*}
\Omega_4(\mathbb{Z}^\mu) &\cong \Omega_4(pt) \oplus H_4(T^\mu;\Z) \\
[\psi \colon Z \to T^\mu] &\mapsto [Z \to \text{pt}] \oplus \psi_* [Z].
\end{align*}
It is therefore enough to show that the signature defect vanishes on the elements of $\Omega_4(\mathbb{Z}^\mu)$ corresponding through the above isomorphism to a set of generators of~$\Omega_4(pt)$ and $H_4(T^\mu;\Z)$.

It is well known that $\Omega_4(pt)$ is generated
by the class of $\mathbb{C}P^2$. As $\mathbb{C}P^2$ is simply connected,
its twisted signature agrees with the untwisted one and consequently its signature defect also vanishes.
Let us pick a product structure~$T^\mu = (S^1)^\mu$ on the torus.
By the K\"unneth formula, the
abelian group~$H_4(T^\mu; \Z)$ is generated
by the fundamental classes of the subtori~$T^4 =(S^1)^4\subset T^\mu$ given
by inclusions of factors. For every homology class $i_*([T^4])\subset H_4(T^\mu; \Z)$, the corresponding element in $\Omega_4(\mathbb{Z}^\mu)$ is the cobordism class
$[i \colon T^4 \to T^\mu]$. The ordinary signature of $T^4$ is immediately seen to vanish. To compute the twisted signature,  consider the coefficient system~$\C^\omega$ on
$T^4 = T^3 \times S^1$.
As $\omega \in \T^\mu$, this coefficient system is non-trivial on the $S^1$-factor.
Consequently, the twisted chain complex is acylic~\cite[Corollary App.B.B]{Viro09} and
$H_2(T^3 \times S^1; \C^\omega) = 0$. Thus, the twisted signature vanishes, and as a consequence the signature defect of the cobordism class $[i\colon T^4\to T^\mu]$ is $0$.
We deduce that the signature defect vanishes on all of~$\Omega_4(\Z^\mu)$.
\end{proof}

\begin{corollary}\label{cor:bordism}
Let $M$ be an oriented $3$-manifold with a map $H_1(M;\Z)\to \Z^\mu$ and let $W$, $W'$ be two fillings of $M$ over $\Z^\mu$. Then,
$\dsign_\omega W= \dsign_\omega W'$ for all $\omega\in \T^\mu$.
\end{corollary}
\begin{proof}Define the closed oriented $4$-manifold $Z:=W\cup_{M} - W'$, and notice that the map to $\Z^\mu$ can be extended to $Z$. Thanks to Proposition~\ref{prop:bordism}, we have $\sign_\omega Z - \sign Z=0$, and by Novikov additivity we get
	\[ 0 =\sign_\omega Z - \sign Z =  (\sign_\omega W -\sign W) -  (\sign_\omega W' -\sign W').\]
\end{proof}

\subsection{Novikov-Wall additivity of the signature}\label{sub:NovikovWall}
A theorem of Wall~\cite{Wall} computes the correction term to the additivity of the signature under the union of two manifolds along a common codimension $0$ submanifold of their boundaries, generalizing Novikov additivity.
We recall Wall's theorem in the case where the correction term vanishes.
\medbreak

Consider an oriented compact~$4$-manifold~$W$ together with an oriented, properly embedded $3$-manifold~$M$,
which separates~$W$ into two pieces~$W_\pm$. Put differently, $W = W_+ \cup_M (-W_-)$ is obtained by
gluing~$W_+$ to $-W_-$ along the submanifold~$M$. Note that $M$ is allowed to have nonempty
boundary~$\Sigma = \partial M \subset \partial W$ itself. This decomposition induces a decomposition
of the boundaries $\partial W_+ =  N_+ \cup_\Sigma -M$ and $\partial W_- =  N_-\cup_\Sigma -M $;
see Figure~\ref{fig:NovikovWall}. From this, we obtain a decomposition of the boundary~$\partial W = N_+ \cup_\Sigma (- N_-)$.
We equip then $\Sigma$ with the
orientation~$\Sigma = \partial M = \partial N_+ = \partial N_-$.

\begin{figure}[ht]
\includegraphics[width=6cm]{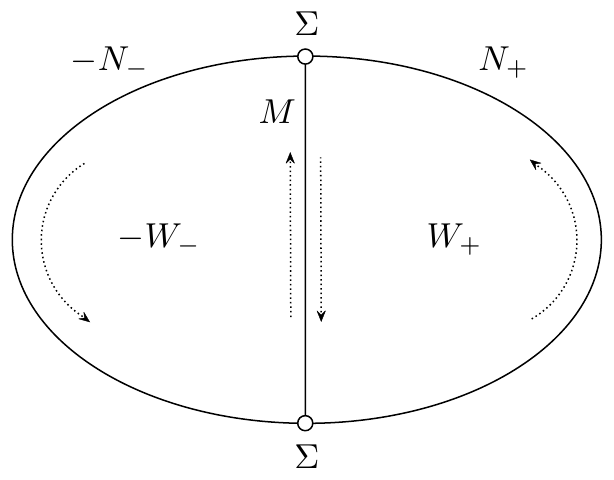}
\caption{A $2$-dimensional sketch of the Novikov-Wall set-up.}
\label{fig:NovikovWall}
\end{figure}

For a manifold $X$ with boundary $\Sigma$, define
\[
V_X:= \ker (H_1(\Sigma;\R)\to H_1(X  ;\R)).
\]
In our setting, we are interested in the spaces $V_M$, $V_{N_+}$ and $V_{N_-}$. The following result is immediately obtained from  the main theorem of~\cite{Wall}, as the correction term vanishes as soon as two of the involved subspaces coincide.
\begin{theorem}
		\emph{(Novikov-Wall additivity)}\label{thm:Wall}
Let $W$ be decomposed as above as the union of $W_+$ and $-W_-$, and suppose that any two among  $V_M$, $V_{N_+}$ and $V_{N_-}$ are equal. Then
	\[	\sign(W)=\sign W_+ - \sign W_- .\]
\end{theorem}

Theorem~\ref{thm:Wall} admits a generalization to twisted coefficients.
For simplicity, in the twisted setting we shall only state a weaker result which is sufficient for our purposes. Suppose to have a map $H_1(W;\Z)\to \Z^\mu$. With this map, we can construct the
local coefficient systems~$\mathbb{C}^\omega$ for every $\omega \in \T^\mu$, as explained in Example~\ref{ex:Ccoefficients}. The following additivity result holds for the twisted signature.
\begin{proposition} \label{prop:TwistedWall}
	Suppose that $W$ is decomposed as above as the union of $-W_-$ and $W_+$.
	Then, for each $\omega\in \T^\mu$ such that $H_1(\Sigma;\C^\omega)=0$,
	Novikov-Wall additivity holds for the twisted signature:
	\[ \sign_\omega W= \sign_\omega W_+ -\sign_\omega W_-.\]
\end{proposition}

\subsection{Concordance roots and vanishing results}\label{sub:ConcordanceRoots}
We generalize the concept of Knotennullstellen~\cite{NagelPowell}.
After applying this concept to a variation of a
well-known chain homotopy argument, we discuss some further properties of these
elements.
\medbreak
Let $U \subset \Z[t_1^{\pm 1},\dots ,t_\mu^{\pm 1}]$ be
the subset of Laurent polynomials~$p(t_1,\dots , t_\mu)$ such that $p(1,\dots ,1)=\pm 1$.
We abbreviate the Laurent ring~$\Z[t_1^{\pm 1},\dots ,t_\mu^{\pm 1}]$ with $\Lambda$.
\begin{definition}
\label{def:ConcordanceRoot}
An element~$\omega \in \T^\mu = (S^1 \sm \{1\})^\mu$ is a \emph{concordance root}
if there is a polynomial $p \in U$ with $p(\omega)=0$. Define $\T^\mu_!$ to be the subset
of all elements $\omega \in \T^\mu$ which are \emph{not} concordance roots.
\end{definition}
Definition~\ref{def:ConcordanceRoot} is a generalization of~\cite[Definition $1.1$]{NagelPowell} to the multivariable case. The key property of non-concordance roots is that they allow us to use a well-known chain homotopy argument~\cite[Proposition 2.10]{CochranOrrTeichner}. The following results are an adaptation of~\cite[Lemma 3.1]{NagelPowell}.

To define the colored (and Alexander) nullity and the colored signature, we will use the bimodules~$\Q(\Z^\mu)$ and $\C^\omega$; see Definition~\ref{def:SignatureNullity} below.
A key ingredient,
necessary to prove the concordance invariance of these invariants, is the following fact: these modules are not
just~$\Lambda$--right modules, but right $U^{-1} \Lambda$-modules where the localisation~$U^{-1} \Lambda$,
inverts all elements of~$U$.

Suppose now that $\Z^m \to \Z^\mu$ is a homomorphism obtained by adding entries.
Then, the induced map of group rings $\Z[\Z^m] \to \Lambda$ fits into the following commutative diagram with the augmentation maps
\[ \begin{tikzcd}
\Z[\Z^m] \ar[rr] \ar[rd, "\aug"']& &\Lambda \ar[ld, "\aug"]\\
& \Z &
\end{tikzcd}. \]
Recall that, the augmentation map sends a Laurent polynomial $p(t_1,\dotsc, t_\mu)$ to its evaluation $p(1,\dotsc , 1)$.
The next lemma follows from considerations of determinants; see cf.~\cite[Proposition 2.4]{CochranOrrTeichner}.
\begin{lemma}\label{lem:DeterminantTrick}
Let $g \colon \Z[\Z^m]^k \to \Z[\Z^m]^k$ be a $\Z[\Z^m]$--module homomorphism with the property
that $\Z \otimes_{\Z[\Z^m]}g$ is an isomorphism. Then
\[ U^{-1} \Lambda \otimes_{\Z[\Z^m]} g \colon  (U^{-1} \Lambda)^k \to (U^{-1} \Lambda)^k \]
is also an isomorphism. Consequently, so is $\Q(\Z^\mu) \otimes_{\Z[\Z^m]}g$ and $\C^\omega \otimes_{\Z[\Z^m]}g$.
\end{lemma}
\begin{proof} See~\cite[Section 3]{NagelPowell}. \end{proof}

\begin{lemma}\label{lem:Cone}
Let $k$ be a non-negative integer, and let $\omega$ lie in $\T_!^\mu$.
If $(X,Y)$ is a pair of CW-complexes over $B\Z^\mu$ with $H_i(X, Y; \Z) = 0$ for $0 \leq i \leq k$,
then both~$H_i(X,Y; \Q(\Z^\mu))$ and~$H_i(X,Y; \C^\omega)$ vanish for~$0 \leq i \leq k$.
\end{lemma}
\begin{proof}
We make the following abbreviations~$C^\Z := C(X,Y; \Z)$ and $C^\Lambda:= C(X,Y; \Lambda)$ for
the cellular chain complexes of the pairs~$(X,Y)$. For the remainder of the proof, $i$ will be an arbitrary integer $0 \leq i \leq k$.
The chain complex~$C^\Z$ consists of finitely generated free $\Z$-modules, and as $H_i(C^\Z) = 0$, it admits a partial
contraction, i.e.\ homomorphisms~$s_i \colon C^\Z_i \to C^\Z_{i+1}$ with
\[ \id_{C_i^\Z} = s_{i-1} \circ d_i + d_{i+1} \circ s_i.\]

Consider the chain map~$\varepsilon \colon C^\Lambda \to C^\Z$ of chain complexes over $\Lambda$,
which is induced by tensoring with the augmentation map. Pick a lift $s^\Lambda_i$ of $s_i$
under $\varepsilon$, which is a homomorphism~$s_i^\Lambda \colon C^\Lambda_i \to C^\Lambda_{i+1}$
of $\Lambda$-modules such that the following diagram commutes:
\[ \begin{tikzcd}
C_i^\Lambda \ar[r, "s^\Lambda_i"] \ar[d, "\varepsilon"] & C_{i+1}^\Lambda \ar[d, "\varepsilon"]\\
C_i^\Z \ar[r, "s_i"] & C_{i+1}^\Z .
 \end{tikzcd}\]
Such a lift exists because $C_i^\Lambda$ consists of free modules and the map~$\varepsilon$
is surjective. Consider the partial chain map
\[ f_i = s^\Lambda_{i-1} \circ d_i + d_{i+1} \circ s^\Lambda_i.\]
By construction,~$\Z \otimes_\Lambda f_i = s_{i-1} \circ d_i + d_{i+1} \circ s_i = id_{{C_i^\Z}}$ and
so $U^{-1} \Lambda \otimes_\Lambda f_i$ is also an isomorphism; see Lemma~\ref{lem:DeterminantTrick}. We obtain
that $U^{-1} \Lambda \otimes_\Lambda s^\Lambda_i$ is a partial chain contraction for
$U^{-1} \Lambda \otimes_\Lambda C^\Lambda$ and
\[ H_i(X,Y; U^{-1} \Lambda) = H_i(U^{-1} \Lambda \otimes_\Lambda C^\Lambda) = 0. \]
Now we tensor with either $R = \Q(\Z^\mu)$ or $R= \C^\omega$, which are both right~$U^{-1} \Lambda$--modules. Here, we use the fact that~$\omega \in \T^\mu_!$. Note that $R \otimes_\Lambda  s^\Lambda_i$ is a partial chain contraction for
$R \otimes_{U^{-1} \Lambda} {U^{-1} \Lambda}\otimes_\Lambda C^\Lambda$
and so $H_i(X,Y; R) = 0$.
\end{proof}

For the remainder of the section, we collect properties of the set~$\T^\mu_!$ of non-concordance
roots.
For a prime~$p$, define
\[ \T_p^\mu := \{ \omega \in \T^\mu \mid \omega_i \text{ is a } p^n\text{-root of unity for some }n \} \]
and $\T_P^\mu := \bigcup_p \T_p^\mu$. This is the set for which concordance invariance properties and genus bounds are proved in~\cite[Section 7]{CimasoniFlorens}. The next result shows that the set~$\T_!^\mu$ of non-concordance roots contains
$\T_P^\mu$.

\begin{proposition}
\label{prop:TPContainedT!}
The set $\T_P^\mu$ is contained in $\T_!^\mu$.
\end{proposition}
\begin{proof}
Let $\omega \in \T_{p}^\mu$ and $q(t_1,\dotsc, t_\mu)$ be a polynomial such that $q(\omega)=0$.
We have to show that $q(1,\dots, 1) \neq \pm 1$.
We pick $n$ large enough such that all
$\omega_i$ are $p^n$-roots of unity.
The subgroup consisting of the $p^n$-roots of unity is cyclic.
Thus we write $\omega=(\zeta^{n_1}, \dots, \zeta^{n_\mu})$ for
a primitive $p^n$-root of unity~$\zeta$.
Define the one variable polynomial $\overline{q}(t):=q(t^{n_1},\dots, t^{n_\mu})$.
Hence, we have $\overline{q}(\zeta)=0$, so $\overline{q}(t)$ is a
multiple of the $p^n$-th cyclotomic polynomial, whose value at $1$ equals $p$.
It follows that $p$ divides $q(1,\dots,1)= \overline{q}(1)$ and so cannot be equal to $\pm 1$.
\end{proof}

The following example shows that $\T^\mu_!$ also contains elements which are not in $\T_P^\mu$, but have algebraic coordinates.
\begin{example}
	\label{ex:T!MoreGeneral}
	We claim that the algebraic element~$\omega=(\frac{3+4i}{5}, -1)$ is in $\T^2_!$, but not contained in $\T_P^2$.
	The algebraic number~$\omega_0 = \frac{3+4i}{5}\in S^1$, has minimal polynomial $p(t)=5t^2-6t+5$
	and is not a root of unity~\cite[Lemma 2.1]{NagelPowell}.
	It follows that $\omega_0$ is not an element of $\T^1_P$.

	To show that $\omega \in \T^2_!$, we prove that any polynomial~$q(t_1, t_2)$ with $q(\omega) = 0$
	has $q(1,1) \neq \pm 1$.
	Consider $\overline{q}(t):=q(t,-1)$ and note that $\frac{3+4i}{5}$ is a root of
	$\overline{q}(t)$. As a consequence $4=p(1)$ divides $\overline{q}(1)$ and
	$\overline{q}(1)=q(1,-1)$ is even. It follows that $q(1,1)$ must also be even.
\end{example}

\begin{lemma}\label{lem:DegenerateEntries}
Let $(\omega_1, \dots, \omega_n) \in \T^n_!$,
and~$\beta \colon \{1, \dots, \mu \} \to \{1, \dots, n\}$ be a map.
Then $(\omega_{\beta(1)}, \dots, \omega_{\beta(\mu)})$ is an element
of $\T^\mu_!$.
\end{lemma}
\begin{proof}
Let $q(t_1, \dots, t_\mu)$ be a polynomial such that $q(\omega_\beta) = 0$,
where $\omega_\beta$ denotes $(\omega_{\beta(1)}, \dots, \omega_{\beta(\mu)})$.
Define a polynomial in $n$-variables by the equality
$p(x_1, \dots, x_n) = q(x_{\beta(1)}, \dots, x_{\beta(\mu)})$.
Note that $ p(\omega_1, \dots, \omega_n) = q(\omega_{\beta(1)}, \dots, \omega_{\beta(\mu)}) = 0$ and as $(\omega_1, \dots, \omega_n) \in \T^n_!$, we deduce that
$q(1,\dots, 1) = p(1,\dots, 1) \neq \pm 1$.
\end{proof}

As shown in the following remark,
it is also easy to construct elements which do not belong to $\T_!^\mu$
and for which our main results will not apply.
\begin{remark}
Let $\omega = (\omega_1, \dots, \omega_\mu) \in \T^\mu$.
A consequence of Lemma~\ref{lem:DegenerateEntries} is that, if $\omega$ belongs to $T_!^\mu$, then
all the coefficients $\omega_i$ belong to $\T_!^1$. Phrasing it differently, if any of the coefficients of $\omega$ is a concordance root, then $\omega$ itself is a concordance root.
\end{remark}

\section{Colored signatures and nullities of links}\label{sec:4dDef}

In Section~\ref{sub:Setup}, we give a definition of the colored signature and nullity of a colored link as twisted invariants of manifolds with boundary.
Section~\ref{sub:CFSignature} shows that they coincide with the invariants introduced by Cimasoni-Florens~\cite{CimasoniFlorens}; see e.g Proposition~\ref{prop:NullityNoCcomplex} and Proposition~\ref{prop:ColorSignature}. Section~\ref{sub:Genus} introduces the notion of colored cobordism and presents the statement of Theorem~\ref{thm:Genus} which provides obstructions on the possible colored cobordisms that two given colored links can bound.
Section~\ref{sub:ProofGenus} is devoted to the proof of the theorem.
Finally, Section~\ref{sub:ApplicationsGenus} provides the applications of Theorem~\ref{thm:Genus} and puts it in relation with some previously known results. In particular, we prove the concordance invariance of the signature and nullity and present obstructions on the possible surfaces a colored link can bound in $D^4$.

\subsection{Set-up}\label{sub:Setup}
This section deals with some preliminaries on colored links and their colored bounding surfaces.
Making use of this set-up, we introduce our main invariants: the colored signature and the colored nullity.
\medbreak
Let $L = L_1 \cup \cdots \cup L_\mu \subset S^3$ be a $\mu$-colored link.
We denote the exterior of~$L$ by $X_L$ and recall that the abelian group~$H_1(X_L; \Z)$ is freely generated by the meridians of $L$. Summing the meridians of the same color, we obtain a homomorphism~$H_1(X_L; \Z) \to \Z^\mu$.
A \emph{colored bounding surface} for a colored link $L$ is a union $F = F_1\cup \cdots \cup F_\mu$ of
properly embedded, locally flat, compact oriented surfaces~$F_i \subset D^4$
with $\partial F_i = L_i$ and which only intersect each other transversally in double points.
A \emph{bounding surface} of a link~$L$ is a union $F = F_1\cup \cdots \cup F_m$ of
properly embedded, locally flat, compact, connected and oriented surfaces~$F_i \subset D^4$
which only intersect each other transversally in double points, and $\partial F = L$.
Note that we require each $F_i$ to be connected. Forgetting about the colors, a colored bounding surface turns into a bounding surface formed by the union of its connected pieces.

As the surfaces $F_i$ are required to be locally flat, that is they admit tubular neighborhoods. Given a (possibly colored) bounding surface~$F$ of $L$, we denote by $\nu F$ the union of some choice of tubular neighborhoods for its components.
We denote then by~$W_F:= D^4 \sm \nu F$ the exterior of $F$. For the convenience of the reader,
we give an argument for the following well-known fact.
\begin{lemma}\label{lem:MayerVietorisExterior}
Given a bounding surface $F$, the abelian group~$H_1(W_F;\Z)$ is freely generated
by the meridians of the components~$F_i$.
\end{lemma}
\begin{proof}
Pick a small ball~$B_x$ around each intersection point~$x$ of $F$.
Note that $W_F = D^4 \sm ( \bigcup_x B_x \cup \bigcup_i \nu F^\circ_i)$, where
the surface~$F^\circ_i$ is $F_i$ with little discs removed around the intersection points.
The Mayer-Vietoris sequence of $D^4 \sm \bigcup_x B_x = W_F \cup \bigcup_i \nu F^\circ_i$
with $\Z$-coefficients gives us
\[ 0 \to  H_1\Big(\bigcup_i (F_i^\circ \times S^1)\Big)
\to H_1\Big(\bigcup_i (F_i^\circ \times D^2)\Big)\oplus H_1(W_F)\to 0,\]
where the $0$'s arise as the homology $H_j\Big(D^4 \sm \bigcup_x B_x\Big)$ for $j=1,2$.
Applying the Künneth theorem to the products $F_i^\circ \times S^1$ the sequence can be reduced to
$0 \to  H_1\big(\bigcup_i \{p_i\} \times S^1; \Z\big) \to H_1(W_F;\Z)\to 0$, where $p_i \in F_i$. This concludes the proof of the lemma.
\end{proof}

Consequently, there is a canonical homomorphism~$H_1(W_F;\Z) \to \Z^\mu$ which restricts to~$H_1(X_L; \Z) \to \Z^\mu$ on the link exterior: indeed the inclusion~$X_L \subset W_F$ sends the meridians of $L$ to the meridians of $F$. Since~$X_L$
and~$W_F$ are now both spaces over $\Z^\mu$, we can give the following definition.
\begin{definition}\label{def:SignatureNullity}
Let $F$ be a colored bounding surface for a $\mu$-colored link $L$. Given $\omega \in \T^\mu$, define the \emph{colored signature}~$\sigma_L(\omega)$ and the
\emph{colored nullity}~$\eta_L(\omega)$ by
\[ \sigma_L(\omega) = \sign_\omega W_F, \quad \eta_L(\omega) = \dim H_1(X_L; \C^\omega).\]
\end{definition}
Viro~\cite[Theorem 2.A]{Viro09} showed that $\sign_\omega W_F$
is independent of the choice of colored bounding surface. For a proof, see also the upcoming paper by Degtyarev, Florens and Lecuona~\cite{DegtyarevFlorensLecuona2}.
It is sometimes useful in the following to view $\sigma_L(\omega)$ as signature defect, which is made possible by the following result, probably well known to the experts.
\begin{proposition} \label{prop:UntwistedSign}
If $F$ is a colored bounding surface for a $\mu$-colored link $L$, the
untwisted intersection form on $W_F$ is trivial. As a consequence, the
signature~$\sign W_F$ vanishes and we have $\sigma_L(\omega) = \dsign_\omega W_F$.
\end{proposition}
\begin{proof}
	Set $M_F:=\overline{\nu F} \cap W_F$ so that $\partial W_F=
X_L\cup_{L\times S^1} M_F$. Consider the portion $ H_2(M_F;\Z)\to
H_2(W_F;\Z)\oplus 0 \to 0$ of the Mayer-Vietoris sequence associated to the
decomposition $D^4=W_F\cup \nu F$. It follows that the map $H_2(M_F;\Z)\to H_2(W_F;\Z)$
is surjective. Since $M_F$ is contained in the boundary of $W_F$,
the natural map $j \colon H_2(\partial W_F;\Z)\to H_2(W_F;\Z)$ is surjective.
The statement follows immediately since elements of $\im j$ annihilate the
intersection form.
\end{proof}

\subsection{C-complexes}
\label{sub:CFSignature}
We recall the multivariable signature and nullity
functions introduced by Cimasoni-Florens~\cite{CimasoniFlorens} using C-complexes. Our main objective is to show that these invariants coincide with the colored signature and nullity defined in Section~\ref{sub:Setup}.
\medbreak
A \emph{C-complex} for a $\mu$-colored link~$L$ consists of a
collection of Seifert surfaces $S_1, \dots , S_\mu$ for the sublinks~$L_1, \dots , L_\mu$
that intersect only along clasps; see~\cite{Cooper, CimasoniPotential, CimasoniFlorens} for details.
Given such a C-complex and a
sequence~$\varepsilon=(\varepsilon_1,\varepsilon_2,\dots, \varepsilon_\mu)$ of $\pm 1$'s,
there are $2^\mu$ \emph{generalized Seifert matrices}~$A^\varepsilon$, which
extend the usual Seifert matrix~\cite{CimasoniPotential, CimasoniFlorens}.
Note that for all~$\eps$,~$A^{-\eps}$ is
equal to~$(A^\eps)^T$. Using this fact, one easily checks that for
any~$\omega = (\omega_1,\dots,\omega_\mu)$ in the~$\mu$-dimensional torus, the matrix
\[
H(\omega)=\sum_\eps\prod_{i=1}^\mu(1-\overline{\omega}_i^{\eps_i})\,A^\eps
\]
is Hermitian. Since this matrix vanishes when one of the coordinates of~$\omega$ is equal to~$1$,
we restrict ourselves to~$\omega \in \T^\mu$. The \emph{multivariable signature} is
the signature of the Hermitian matrix~$H(\omega)$ and the \emph{multivariable nullity} is $\operatorname{null}H(\omega) + \beta_0(S)-1$,
where $\beta_0(S)$ is the number of connected components of $S$.

We start by proving that~$\eta_L(\omega) =\operatorname{null}H(\omega) + \beta_0(S)-1$, i.e that the colored nullity is equal to the multivariable nullity:

\begin{proposition}\label{prop:NullityNoCcomplex}
Let $L$ be a $\mu$-colored link. For every $\omega \in \T^\mu$ and for any C-complex $S$ for $L$, we have the
equality~$\eta_L(\omega) =\operatorname{null}H(\omega) + \beta_0(S)-1$.
\end{proposition}

\begin{proof}
Since the multivariable nullity $\operatorname{null}H(\omega) + \beta_0(S)-1$
is independent of the chosen $C$-complex~\cite[Theorem $2.1$]{CimasoniFlorens},
pick $S$ for which there is at least one clasp between each pairs of surfaces
$S_i$ and $S_j$, so that in particular~$\beta_0(S)=1$. Note that this is
possible thanks to~\cite[Lemma $3$]{CimasoniPotential}. Using~\cite[Corollary
3.6]{CimasoniFlorens} the Alexander module
$H_1(X_L;\Lambda_S)$ admits a square presentation matrix given by $H(t)$.
Tensoring with $\mathbb{C}^\omega$ we deduce that $H(\omega)$ presents
$\mathbb{C}^\omega \otimes_{\Lambda_S} H_1(X_L;\Lambda_S)$. Using
Lemma~\ref{lem:UCSSTorExample}, we obtain that $H_1(X_L;\C^\omega)=\mathbb{C}^\omega
\otimes_{\Lambda_S} H_1(X_L;\Lambda_S)$ and consequently $H(\omega)$ also
presents $H_1(X_L;\mathbb{C}^\omega)$. The result follows immediately.
\end{proof}

We conclude by showing that the colored signature $\sigma_L(\omega)$ coincides with the multivariable signature of Cimasoni-Florens~\cite{CimasoniFlorens}.
\begin{proposition}\label{prop:ColorSignature}
If $L$ is a $\mu$-colored link $L$, then $\sigma_L(\omega) = \sign H(\omega)$, i.e.\ the colored signature is equal to the multivariable signature.
\end{proposition}
\begin{proof}
Since the colored signature is independent of the choice of a colored bounding surface,
we can take $F$ to be a
push-in of a C-complex in the $4$-ball; see~\cite[Section
3.1]{ConwayFriedlToffoli} for a precise description.  By~\cite[Theorem
$1.3$]{ConwayFriedlToffoli}, the intersection pairing $\lambda_{\Lambda_S}$
is represented by $H(t)$. Since we wish to show that the intersection pairing
$\lambda_{\C^\omega}$ is represented by $H(\omega)$, the theorem will
follow if we manage to produce the following commutative diagram
\begin{equation}
\label{eq:Desired}
\xymatrix@R0.4cm{
\mathbb{C}^\omega  \otimes_{\Lambda_S} H_2(W_F;\Lambda_S) \times \mathbb{C}^\omega  \otimes_{\Lambda_S} H_2(W_F;\Lambda_S) \ar[r]\ar[d] \ar[d]& \mathbb{C}^\omega \otimes_{\Lambda_S} \Lambda_S\ar[d] \\
H_2(W_F;\mathbb{C}^\omega ) \times H_2(W_F;\mathbb{C}^\omega ) \ar[r] & \mathbb{C}.
}
\end{equation}
Further assuming $S$ to be totally connected implies that $H_i(W_F;\Lambda_S)$ vanishes for $i \neq 2$, and is a finitely generated free $\Lambda_S$-module for $i=2$~\cite[Section $3$ and Proposition $4.1$]{ConwayFriedlToffoli}.

Consider the following diagram below, where homology groups and tensor products
without coefficients are over $\Lambda_S$.
Applying the universal coefficient spectral sequence, as described in Propositions~\ref{prop:UCSS} and~\ref{prop:UCSSTor}, the first three
vertical maps in the following commutative diagram are isomorphisms
\[
\begin{tikzcd}[column sep=0.3cm]
\mathbb{C}^\omega  \otimes H_2(W_F) \ar[r]\ar[d,"\cong"]
& \mathbb{C}^\omega  \otimes H_2(W_F,\partial W_F) \ar[r]\ar[d,"\cong"]
& \mathbb{C}^\omega  \otimes H^2(W_F) \ar[r] \ar[d,"\cong"]
&\mathbb{C}^\omega  \otimes \text{Hom}_{\Lambda_S}(H_2(W_F),\Lambda_S)^{\text{tr}}
	\ar[d,"\cong"] \\
H_2(W_F;\mathbb{C}^\omega ) \ar[r]
& H_2(W_F,\partial W_F;\mathbb{C}^\omega)\ar[r]
& H^2(W_F;\mathbb{C}^\omega )\ar[r]
& \text{Hom}_{\mathbb{C}}(H_2(W_F;\mathbb{C}^\omega),\mathbb{C})^{\text{tr}}.
\end{tikzcd}
\]
The last vertical map is an isomorphism since $H_2(W_F;\Lambda_S)$ is finitely generated
and free. Considering the adjoint, we precisely obtain the diagram of Equation~(\ref{eq:Desired}).
\end{proof}

\subsection{The genus bound}\label{sub:Genus}

For elements~$\omega \in \T_P^\mu$, the multivariable
signature and nullity are known to give lower bounds on the genus of colored
bounding surfaces~\cite[Theorem 7.2]{CimasoniFlorens}. In this section we prove
a more general result for surfaces in~$S^3 \times [0,1]$.  As corollaries, we extend the concordance invariance results of~\cite[Theorem 7.1]{CimasoniFlorens}
and generalize the lower bounds of~\cite[Theorem 7.2]{CimasoniFlorens}.

\begin{definition}\label{def:cobordism}
A \emph{colored cobordism}
between two $\mu$-colored links $L$ and $L'$ is a
collection of properly embedded locally flat surfaces~$\Sigma = \Sigma_1 \cup \dots \cup \Sigma_\mu$
in $S^3 \times [0,1]$ that have the following properties: the surfaces only intersect each other in double points,
	each surface~$\Sigma_i$ has boundary~$-L_i \sqcup L_i'$, and each connected component of $\Sigma_i$ has a boundary both in $S^3 \times \{0\}$ and in $S^3 \times \{1\}$. We say that $\Sigma$ \emph{has $m$ components} if the disjoint union of the surfaces $\Sigma_1,\dotsc, \Sigma_\mu$ has $m$ connected components.
\end{definition}

The main result of this section is the following lower bound.
\begin{theorem}\label{thm:Genus}
If $\Sigma = \Sigma_1 \cup \dots \cup \Sigma_\mu$
is a colored cobordism between two $\mu$-colored links $L$ and $L'$ with $c$ double points, then
\[|\sigma_L(\omega)-\sigma_{L'}(\omega)|  + |\eta_L(\omega)-\eta_{L'}(\omega)|
	\leq \sum_{i=1}^{\mu} -\chi(\Sigma_i) + c \]
for all $\omega\in \T_!^\mu$.
\end{theorem}
\begin{remark}\label{rem:formulas}
The right-hand side of the inequality can equivalently be expressed in terms of the first Betti number or of the genus of the surfaces. Suppose that $L$ is an $n$-component link, $L'$ is an $n'$-component link, and that the cobordism $\Sigma$ has $m$ components
(in the sense of Definition~\ref{def:cobordism}). Then, we have the following equalities:
\[ \sum_{i=1}^{\mu} -\chi(\Sigma_i) +c =  \sum_{i=1}^{\mu}b_1(\Sigma_i) -m+c = \sum_{i=1}^{\mu} 2 g_i(\Sigma_i)+ n + n'-2m +c.\]
For this reason, we will usually refer to the inequality of Theorem \ref{thm:Genus} as a \emph{genus bound}, even if the genus does not appear explicitly in the formula.
\end{remark}

\subsection{Proof of the main theorem}\label{sub:ProofGenus}
We proceed towards the proof of Theorem~\ref{thm:Genus}, starting with a series of preliminary results.

\medbreak

First, we describe the Euler characteristic of the exterior~$W_\Sigma$ of a
colored cobordism~$\Sigma$ in $S^3 \times [0,1]$ in terms of the Euler characteristic of the surfaces $\Sigma_i$.

\begin{lemma}\label{lem:eulersigma}
Suppose $\Sigma$ is a $\mu$-colored cobordism between two colored links~$L$ and~$L'$ with $c$ double points. Then the Euler characteristic of $W_\Sigma$ is given by
\[\chi(W_\Sigma) = \sum_{i=1}^\mu -\chi(\Sigma_i) +c .\]
\end{lemma}
\begin{proof}
First, we prove that~$\chi(W_\Sigma)=-\chi(\nu \Sigma)$.
Consider the decomposition~$S^3 \times I = \nu \Sigma \cup W_\Sigma$ and set $M_\Sigma := \nu \Sigma \cap W_\Sigma$. Using the decomposition formula for the Euler characteristic yields
$\chi(S^3 \times I) = \chi(W_\Sigma) + \chi(\nu \Sigma ) -  \chi(M_\Sigma)$. As the Euler characteristic of a $3$-manifold with toroidal boundary vanishes, $\chi(M_\Sigma)=0$. Since $\chi(S^3 \times I)$ also vanishes, the claim follows.  Now note that $\nu \Sigma$ is homotopy equivalent to the union~$A = \bigcup_i \Sigma_i \subset S^3$.
Recall that the surfaces~$\Sigma_i$ intersect each other in $c$ points.
We apply again the decomposition formula for $A$ and obtain
\[ \chi(A) = \sum_{i=1}^\mu \chi(\Sigma_i)
-\chi\Big( \bigcup_{i \neq j} \Sigma_i \cap \Sigma_j \Big)= \sum_{i=1}^\mu \chi(\Sigma_i) -c. \]

\end{proof}

By Lemma~\ref{lem:MayerVietorisExterior}, one observes that $H_1(W_\Sigma;\Z)$ is freely generated by the meridians of $\Sigma$. Consequently, there is a homomorphism $H_1(W_\Sigma;\Z) \to \Z^\mu$ that extends the maps on $H_1(X_L;\Z)$ and $H_1(X_{L'};\Z)$.

Next, we observe that with $\C^\omega$ coefficients, the boundary of $W_\Sigma$
behaves as the disjoint union of the link exteriors~$X_L$ and $X_{L'}$.
\begin{lemma}
\label{lem:splitboundary}
	The inclusion of $X_L\sqcup X_{L'}$ into $\partial W_\Sigma$ induces an isomorphism
	\[H_i(X_L;\C^\omega) \oplus 	H_i(X_{L'};\C^\omega) \cong H_i(\partial W_\Sigma;\C^\omega)\]
for all $\omega \in \T^\mu$.
\end{lemma}
\begin{proof}
The boundary of $W_\Sigma$ decomposes into the union of $X_L$, $X_{L'}$ and the
plumbed $3$-manifold~$M_\Sigma$.
The homology groups~$H_*(M_\Sigma;\Lambda_S)$ are zero~\cite[Proof of Lemma 5.2]{ConwayFriedlToffoli}.
The universal coefficient spectral sequence of Proposition~\ref{prop:UCSSTor}
implies that $H_*(M_\Sigma;\C^\omega)=0$.
The result now follows from the Mayer-Vietoris exact sequence for $\partial W_F$.
\end{proof}

The next lemma provides some information on the twisted homology of $W_\Sigma$.

\begin{lemma}\label{lem:034AndNullityBound}
If $\Sigma\subset S^3\times I$ is a $\mu$-colored cobordism between $L$ and $L'$ and $\omega \in \T_!^\mu$, then
\begin{enumerate}
\item $\beta_1^\omega(W_\Sigma) \leq \eta_L(\omega)$
		and $\beta_1^\omega(W_\Sigma) \leq \eta_{L'}(\omega)$,
\item $H_{i}(W_\Sigma; \C^\omega) = 0$ for $i= 0,3,4$.
\end{enumerate}
\end{lemma}
\begin{proof}
As $W_\Sigma$ and $X_L$ are both connected, there is an isomorphism $H_0(X_L;
\Z)\cong H_0(W_\Sigma; \Z)$. Since the inclusion $X_L \subset W_\Sigma$ takes meridians to meridians, $H_1(X_L;\Z)\to H_1(W_\Sigma;\Z)$ is surjective. Combining these facts, $H_i(W_\Sigma,X_L;\Z)=0$, so that Lemma~\ref{lem:Cone} gives $ H_i(W_\Sigma,X_L; \C^\omega)=0$ for
$i=0,1$. It follows from the long exact sequence of the pair $(W_\Sigma,X_L)$
that the inclusion induced map $H_1(X_L;\C^\omega)\to H_1(W_\Sigma;\C^\omega)$ is surjective, and thus
$\beta_1^\omega(W_\Sigma) \leq \eta_L(\omega)$. Repeating the argument for $X_{L'}$, the
first statement is proven.

Since the inclusion of $X_L$ into $W_\Sigma$ factors through $\partial W_\Sigma$, an analogous argument shows that $H_i(W_\Sigma,\partial W_\Sigma; \C^\omega)=0$ for $i=0,1$. Lemma~\ref{lem:DualityUCSS} now
implies that $H_i(W_\Sigma; \C^\omega)=0$ for $i=3,4$.

Note that the entries of~$\omega = (\omega_1, \ldots, \omega_\mu)$ are different from~$1$. This implies that the vector space~$H_0(W_\Sigma; \C^\omega)$ vanishes by its description as a quotient~\cite[Section VI.3]{Hilton97}.
\end{proof}

We conclude this section with a dimension count, which will prove itself useful to
bound the twisted signature of $W_\Sigma$.
\begin{lemma}\label{lem:EqGenus}
	Denote by $j \colon H_2(\partial W_\Sigma;\C^\omega) \to H_2(W_\Sigma;\C^\omega)$
	the map induced by the inclusion. Then, for $\omega \in \T^\mu_!$, we have
\[ \dim(\coker j)
= \beta_2^\omega(W_\Sigma)-\beta_2^\omega(\partial W_\Sigma) +\beta_1^\omega(W_\Sigma), \]
\end{lemma}
\begin{proof}
Recall that by Lemma~\ref{lem:034AndNullityBound}, the vector space $H_3(W_\Sigma;\C^\omega)$
vanishes.
Consider the following portion of the long exact sequence of the pair~$(W_\Sigma,\partial W_\Sigma)$:
\[ 0\to H_3(W_\Sigma, \partial W_\Sigma; \C^\omega)\xrightarrow{\delta}
	H_2(\partial W_\Sigma; \C^\omega)\xrightarrow{j} H_2( W_\Sigma;\C^\omega).\]
By exactness, $\dim(\im j)=\beta_2^\omega(\partial W_\Sigma) - \dim(\im \delta)$.
As $\delta$ is injective, one gets
$ \dim (\im j) =  \beta_2^\omega(\partial W_\Sigma)- \beta_3^\omega(W_\Sigma, \partial W_\Sigma)$. The result now follows since Lemma~\ref{lem:DualityUCSS} implies that
$\beta_3^\omega(W_\Sigma, \partial W_\Sigma)= \beta_1^\omega(W_\Sigma)$.
\end{proof}

We are now ready to conclude the proof of Theorem~\ref{thm:Genus}.
\begin{proof}[Proof of Theorem~\ref{thm:Genus}]
We start by proving the following inequality:
\[ |\sign_\omega W_\Sigma| \leq \chi(W_\Sigma)-|\eta_L(\omega)-\eta_{L'}(\omega)|. \]
As in Lemma~\ref{lem:EqGenus}, we use $j$ to denote the map~$H_2(\partial W_\Sigma; \C^\omega) \to H_2(W_\Sigma; \C^\omega)$.
Since the twisted intersection form~$\lambda_{\C^\omega}$ descends to a pairing on
$H_2(W_\Sigma;\C^\omega)/ \im j$, an application of Lemma~\ref{lem:EqGenus} yields
\begin{equation}\label{eq:FirstStepGenus}
|\sign_\omega W_\Sigma| \leq \dim \frac{H_2(W_\Sigma;\C^\omega)}{\im j} = \beta_2^\omega(W_{\Sigma})-\beta_2^\omega(\partial W_{\Sigma}) +\beta_1^\omega(W_{\Sigma}).
\end{equation}
Now, thanks to Lemma~\ref{lem:034AndNullityBound}, we have
$\chi(W_\Sigma) = \beta_2^\omega(W_\Sigma) - \beta_1^\omega(W_\Sigma)$, and using Lemma~\ref{lem:splitboundary}, one gets
$\beta_1^\omega(\partial W_\Sigma)=\eta_L(\omega)+ \eta_{L'}(\omega)$.
Using these last two identities, Equation (\ref{eq:FirstStepGenus}) can be rewritten as
\[ |\sign_\omega W_{\Sigma}| \leq \chi(W_\Sigma) + 2\beta_1^\omega(W_\Sigma) - \eta_L(\omega)- \eta_{L'}(\omega). \]
The desired inequality is now obtained by using Lemma~\ref{lem:034AndNullityBound} to bound $\beta_1^\omega(W_\Sigma)$
above both by $\eta_L(\omega)$ and $\eta_{L'}(\omega)$.

With the inequality above, Theorem~\ref{thm:Genus}
will follow from Lemma~\ref{lem:eulersigma} once we have established that
\[ \sign_{\omega} W_\Sigma   =  \sign_{L'}(\omega)-\sign_L(\omega). \]
Pick a colored bounding surface $F\subset D^4$ for $L$. Thanks to Proposition~\ref{prop:UntwistedSign}, we have $\sigma_L(\omega)=\sign_\omega(W_F)$. One can now form the surface with singularities $F':=F\cup_L \Sigma \subset D^4\cup_{S^3} S^3\times I$. Using an orientation-preserving diffeomorphism between $D^4\cup_{S^3} S^3\times I$ and $D^4$, the surface $F\cup_L \Sigma$ is sent to a colored bounding surface for~$L'$.
Its exterior~$W_{F'}$ is clearly homeomorphic to $W_F\cup_{X_L} W_\Sigma$.
Once again thanks to Proposition~\ref{prop:UntwistedSign}, we have $\sigma_{L'}(\omega)=\sign_\omega(W_{F'})$.
Since $H_1(L\times S^1; \C^\omega)=0$, Proposition~\ref{prop:TwistedWall} implies that
Novikov additivity holds for the twisted signature, yielding
\[  \sign_\omega W_{F'} = \sign_\omega W_F + \sign_\omega W_\Sigma. \]
Summarizing, we have shown that
$\sigma_{L'}(\omega)= \sigma_L(\omega) + \sign_\omega W_\Sigma$. Combining this with the inequality of Equation (\ref{eq:FirstStepGenus}) concludes the proof of Theorem~\ref{thm:Genus}.
\end{proof}

\subsection{Applications of the genus bound}\label{sub:ApplicationsGenus}
We will give two applications of Theorem~\ref{thm:Genus}.
First, we show that the colored signature and nullity are concordance invariants,
see Corollary~\ref{cor:ConcordanceViaGenus}, then we study the genus of
colored bounding surfaces in Corollary~\ref{cor:CimasoniFlorens72}.
\medbreak
Two $\mu$-colored links~$L$ and~$L'$ are \emph{concordant}
if there exists a $\mu$-colored cobordism between $L$ and $L'$ which has
no intersection points and consists exclusively of annuli.

\begin{corollary}
\label{cor:ConcordanceViaGenus}
If $L$ and $L'$ are two colored links that are concordant, then
\[ \sigma_{L}(\omega)=\sigma_{L'}(\omega) \quad \quad\text{and} \quad \quad \eta_{L}(\omega)=\eta_{L'}(\omega)\]
for all $\omega \in \mathbb{T}^\mu_!$.
\end{corollary}
\begin{proof}
We apply Theorem~\ref{thm:Genus} to the case where each $\Sigma_i$ is a union of annuli and  there are no double points.
The result follows as all the terms in the right-hand side of the inequality are zero.
\end{proof}
\begin{remark}
In~\cite[Theorem 7.1]{CimasoniFlorens}, Corollary~\ref{cor:ConcordanceViaGenus} is stated for all $\omega \in \T_P^\mu$. However, note that~\cite[Proposition 2.3]{NagelPowell} presents two $1$-colored links $L$ and~$L'$, which have the property that $\sigma_L(z) = \sigma_{L'}(z)$ and $\eta_L(z) = \eta_{L'}(z)$ for all $z \in \mathbb{T}^{1}_P$, but such that there exists a $z_0 \in \mathbb{T}^{1}_!$ with $\sigma_L(z_0) \neq \sigma_{L'}(z_0)$ and $\eta_L(z_0) \neq \eta_{L'}(z_0)$.
\end{remark}

Note that Corollary~\ref{cor:ConcordanceViaGenus} will be significantly improved upon in Section~\ref{sec:Solvable}: the signature and nullity will be shown to be invariant under $0.5$-solvable cobordisms.

Using $\beta_1(F)$ to denote the first Betti number of a surface~$F$, an application of Theorem~\ref{thm:Genus} also gives the inequality below.

\begin{corollary} \label{cor:CimasoniFlorens72}
Let $F=F_1\cup\cdots \cup F_\mu$ be a colored bounding surface for $L$, and suppose that $F$ has $m$ components and $c$ intersection points. Then, for all $\omega \in \T_!^\mu$, we have
\[  |\sigma_L(\omega)|+|\eta_L(\omega)-m+1| \leq \sum_{i=1}^\mu \beta_1(F_i) +c. \]
\end{corollary}
\begin{proof}
Remove small $4$-balls in the interior of $D^4$ on each component of $F$. With
small enough balls,~$F$ will intersect the boundary spheres in unknots.
Tubing the boundary spheres together, we have constructed a $\mu$-colored
cobordism~$\Sigma$ with $m$ components between $L$ and a $\mu$-colored unlink $L'$ of
$m$ components. Thanks to the results of Section~\ref{sub:CFSignature}, we can compute the signature and
nullity of $L'$ using C-complexes~\cite[Section 2]{CimasoniFlorens}.
We pick a disjoint union of $m$ disks as a C-complex. The resulting generalized Seifert matrices are
empty, yielding $\sigma_{L'}(\omega)=0$ and $\eta_{L'}(\omega) = 0+\beta_0(S)-1=m-1$ for all $\omega\in \T^\mu$.
Using Theorem~\ref{thm:Genus} and Remark~\ref{rem:formulas}, we get
\[|\sigma_L(\omega)|+|\eta_L(\omega)-m+1| \leq - \sum_{i=1}^\mu \chi(\Sigma_i)
	+c=\sum_{i=1}^{\mu}b_1(\Sigma_i) -m+c.\]
Now, if $C$ is any of the $m$ components of $F$, the corresponding component
$C'$ of $\Sigma$ is obtained from $C$ by removing a small disk, so that
$\beta_1(C')=\beta_1(C)+1$. Summing over all the components, we get
$\sum_{i=1}^{\mu}\beta_1(\Sigma_i)=\sum_{i=1}^{\mu}\beta_1(F_i)+m$, whence the desired
formula.
\end{proof}

The next example discusses the (non)-sharpness of the bound of Corollary~\ref{cor:CimasoniFlorens72}.
\begin{example} 
We start with an example where the bound is sharp.
Consider the $1$-colored Hopf link~$H = L_1 = K_0 \cup K_1$ (with any orientation). The oriented link~$H$ bounds an annulus~$A$ in $S^3$, and we compute~$|\sigma_H(-1)| = 1$ and $\eta_H(-1) = 0$. If we push $A$ into the $4$--ball, we obtain a bounding surface~$F = F_1 = A$. The inequality of Corollary~\ref{cor:CimasoniFlorens72} is sharp:
	\[ 1 = |\sigma_H(-1)|+|\eta_H(-1)-1+1| \leq \beta_1(A) +0 = 1. \] 
Although it is easy to construct examples
where this inequality is not sharp, 
we claim that the defect can in fact be arbitrarily large: pick a family of knots~$J_n$ such that~$J$ has the Seifert matrix of a slice knot, and topological $4$--genus~$g^\text{top}_4(J_n) \geq~n$ (such knots exist thanks to~\cite[Theorem 1.3]{Cha08}). Now consider $H(J_n) = K_0 \cup (K_1 \# J_n)$, where we tie the knot~$J_n$ into $K_1$ in a small $3$--ball disjoint from~$K_0$. The signature~$\sigma_{H(J_n)}(-1) = 1$ and the nullity~$\eta_{H(J_n)}(-1) = 0$ do not change, but we have~$g_4(H(J_n)) \geq g_4(J_n) -1 \geq n-1$, concluding the proof of the claim.

Instead, if we pick each knot $J_n$ to be topologically slice, but with smooth $4$--genus~$g_4^\text{smooth}(J_n) \geq n$ (such knots exist~\cite[Remark 1.2]{Tanaka98}), then the $H(J_n)$ provide a family of knots where the inequality is sharp in the topological category, but not in the smooth category.
\end{example}

We now compare Corollary~\ref{cor:CimasoniFlorens72} with previous results.

\begin{remark}
\label{rem:Genus}
Corollary~\ref{cor:CimasoniFlorens72} is a generalization of~\cite[Theorem 7.2]{CimasoniFlorens}.
In that paper, it is proven in the smooth setting and requires~$\omega$ to be in the set $\T_{P}^\mu$,
which is strictly smaller as $\T_!^\mu$; see Example~\ref{sub:ConcordanceRoots}.
Since all surfaces $F_i$ are assumed to be connected, $\mu$ appears instead of $m$ in their formula.

We can also recover a previous result~\cite[Theorem 1.4]{Powell} bounding the $4$-genus
of a $1$-colored link~$L$ with $m$ components.
Consider disjoint surfaces $F_1,\dots, F_m$ in $D^4$ bounding $L$.
Indeed, $F:=F_1\sqcup\cdots \sqcup F_m$ is a $1$-colored bounding surface for $L$, and applying
Corollary~\ref{cor:CimasoniFlorens72}, we get
	\[|\sigma_L(\omega)|+|\eta_L(\omega)-m+1 | \leq  \beta_1(F)=2g(F), \]
for $\omega\in \T_!^1$.
The result follows by passing to the minimum over all such collections of surfaces and observing that $\T_!^1$ is dense in $S^1$.

Finally, note that Viro proves inequalities similar to Corollary~\ref{cor:CimasoniFlorens72} in any odd dimension. In particular, for links in $S^3$ he obtains $|\sigma_L(\omega)|+\eta_L(\omega) \leq \beta_2(F,L)+\beta_1(F)$ and  $|\sigma_L(\omega)|+\eta_L(\omega) \leq \beta_1(F,L)+\beta_0(F)$~\cite[Theorem 4.C]{Viro09}. Reworking his equations leads to the inequality
$$|\sigma_L(\omega)|+\eta_L(\omega)-m \leq \sum_{i=1}^\mu \beta_1(F_i) +c,$$
which is slightly weaker than Corollary~\ref{cor:CimasoniFlorens72}. The interested reader will note that while Viro essentially obtains his results for all $\omega \in \T^\mu_!$, his methods are quite different from the chain homotopy argument we rely on, see~\cite[Appendix C]{Viro09}.
\end{remark}

\section{Plumbed \texorpdfstring{$3$--manifolds}{3-manifolds} and surfaces in the \texorpdfstring{$4$--ball}{4-ball}}\label{sec:Plumbed}
In this section, we review plumbed $3$-manifolds and prove a vanishing result for their signature defect. This result is a key step in the proof of Theorem \ref{thm:SolvableNullitySignature} (which is concerned with the invariance of the signature and nullity under $0.5$-solvable cobordisms).
In Section~\ref{sub:APS}, we show this vanishing result in the case of
products of a closed surface with $S^1$.
To do so, we apply a product formula for the Atiyah-Patodi-Singer rho invariant, and pass from the smooth to the topological setting by using a bordism argument. In Section~\ref{sub:Plumbing}, we introduce the framework of plumbed $3$-manifolds and prove the main result, which is contained in Proposition~\ref{prop:PbFilling}. This proposition shows that the signature defect of a $4$-manifold vanishes if its boundary is a  so-called ``balanced" plumbed $3$-manifold.
Finally, in Section~\ref{sub:SurfacesD4} we describe how plumbed $3$-manifolds arise naturally from surfaces intersecting transversally in the $4$-ball, and we perform a homological computation which is needed in Section~\ref{sec:Solvable}.

\subsection{The rho invariant of a product \texorpdfstring{$\Sigma\times S^1$}{Sigma x S1}}\label{sub:APS}

We consider the \emph{rho invariant}~$\rho(M,\alpha)$, a real number,
in the special case of $M$ being a smooth, odd dimensional manifold
with a homomorphism~$\alpha\colon H_1(M;\Z)\to U(1)$~\cite{AtiyahPatodiSinger}.
The definition of the rho invariant requires spectral analysis
of elliptic differential operators on a manifold,
and we will not attempt to recall it.
Instead we state the following properties of $\rho(M,\alpha)$, which
will be sufficient for the purposes of this article.

\begin{proposition}\label{prop:APS}
~
\begin{enumerate}
\item\label{Item:Bounding} If $Z$ is a smooth $2n$-manifold together with a
homomorphism~$\alpha\colon H_1(Z;\Z)\to~S^1$,
then $\rho(\partial Z,\alpha) = -(\sign_\alpha Z - \sign Z)$.

\item\label{Item:Tensor} If $N$ is a closed smooth $2m$-manifold with a homomorphism $\alpha\colon H_1(N;\Z)\to~S^1$,
and $S^1$ comes with a homomorphism $\beta\colon H_1(S^1;\Z)\to S^1$,
then
\[ \rho(N\times S^1, \alpha\otimes \beta) 
		=(-1)^m \sign N \cdot \rho(S^1, \beta).\]
In particular, $\rho(N\times S^1, \alpha\otimes \beta) =0$ if $m$ is odd.
\end{enumerate}
\end{proposition}
\begin{proof}
The first result is the specialization to our setting of the Atiyah-Patodi-Singer index theorem~\cite[Theorem 2.4]{AtiyahPatodiSinger}. The formula in the second statement follows from a direct computation combined with the classical Atiyah-Singer theorem.
Both results can be found in~\cite[Theorem 1.2, (iii) and (v)]{Neumann79}, where it has to be observed that the invariant considered by the author differs from the rho invariant by a sign and that $\sign N = \sign_\alpha N$  (this follows from~(\ref{Item:Bounding}) since $N$ has no boundary, or
alternatively from the Hirzebruch signature formula; see Remark~\ref{rem:SignatureRemark}). The last claim follows immediately from the fact that the ordinary signature of a closed manifold is non-trivial only in dimension $4k$.
\end{proof}

We restrict further to manifolds~$M$ with a homomorphism~$H_1(M;\Z) \to \Z^\mu$.
Since one-dimensional representations of $H_1(M;\Z)$ factoring through $\Z^\mu$ are in bijection
with values~$\omega\in (S^1)^\mu$, we will denote by
$\rho_\omega(M)$ the rho invariant corresponding to the representation~$\alpha$
given by the composition
\[\alpha \colon H_1(M;\Z) \to \Z^\mu \xrightarrow{\omega} S^1.\]
Using Proposition~\ref{prop:APS}, we prove the following lemma.
\begin{lemma}\label{lem:EtaVanishing}
	If $\Sigma$ is a closed oriented connected surface and $\phi \colon H_1(\Sigma \times S^1; \Z) \to \Z^\mu$ is a homomorphism, then $\rho_\omega(\Sigma \times S^1) = 0$ for all $\omega \in \T^\mu$.
\end{lemma}
\begin{proof}
	Since $H_1(\Sigma \times S^1; \Z) = H_1(\Sigma;\Z) \oplus H_1(S^1; \Z)$,
we may restrict $ \phi \colon H_1(\Sigma \times S^1; \Z) \to \Z^\mu$ to
each summand. This produces maps $\phi_\Sigma \colon H_1(\Sigma; \Z) \to \Z^\mu$
and $\phi_{S^1} \colon H_1( S^1; \Z) \to \Z^\mu$. Postcomposing each of
these maps with the map $\Z^\mu \xrightarrow{\omega} S^1$ produces maps~$\varphi, \varphi_\Sigma$
and $\varphi_{S^1}$. Since these maps fit in the commutative diagram
	\[
	\begin{tikzcd}[column sep=1.5cm]
	H_1(\Sigma \times S^1; \Z) \ar[d, "pr_\Sigma \oplus pr_{S^1}"] \ar[r, "\varphi"] & S^1\\
	H_1(\Sigma; \Z) \oplus H_1(S^1; \Z) \ar[r, "\varphi_\Sigma \times \varphi_{S^1}"]
	& S^1 \times S^1, \ar[u,"\cdot"]
	\end{tikzcd}
	\]
it follows that $\varphi = \varphi_\Sigma \otimes \varphi_{S^1}$. Using point~(\ref{Item:Tensor}) of Proposition~\ref{prop:APS}, one obtains
\[\rho_\omega( \Sigma \times S^1 ) = \rho(\Sigma \times S^1, \varphi_\Sigma\otimes \varphi_{S^1})
		=  0.
\]
This concludes the proof of the lemma.
\end{proof}

The following corollary is nearly immediate.
\begin{corollary}\label{cor:twSignatureVanishing}
	Let $V$ be a set of closed oriented connected surfaces.
	If $W$ is a $4$-manifold over $\Z^\mu$ with boundary
	\[ \partial W = \bigsqcup_{\Sigma \in V} \Sigma \times S^1,\]
	then $\sign_\omega W - \sign W= 0$.
\end{corollary}

\begin{proof}
Thanks to point (\ref{Item:Bounding}) of Proposition~\ref{prop:APS},
the number $\sign_\omega W - \sign W$ coincides with minus the rho invariants of its boundary.
By Lemma~\ref{lem:EtaVanishing} and additivity of the rho invariant under disjoint union of manifolds~\cite[Theorem $1.2.1$]{Neumann79},
we get $\sign_\omega W - \sign W=0$.
\end{proof}
Since Proposition~\ref{prop:APS} required the cobounding manifold to be smooth, one might worry
about Corollary~\ref{cor:twSignatureVanishing} only holding for smooth $4$-manifolds~$W$.
The following remark deals with this issue.
\begin{remark}\label{rem:TOP}
Let $W$ be a topological $4$-manifold bounding~$M = \bigsqcup_{\Sigma \in V} \Sigma \times S^1$.
The bordism groups are computed in both the topological
case and the smooth case by $\Omega_3(\Z^\mu) = H_3(\Z^\mu; \Z)$.
Thus, if $M$ bounds topologically, then there also exists a smooth filling~$W'$, for which the rho invariant computation gives $\sign_\omega W' -\sign W'=0$. By
Corollary~\ref{cor:bordism}, the difference between twisted and ordinary
signature is the same for two $4$-manifolds filling the same $M$ over $\Z^\mu$,
so we conclude that $\sign_\omega W -\sign W$ is also zero as desired.
\end{remark}

\subsection{Plumbings and their signature defect}\label{sub:Plumbing}
After reviewing the definition of a plumbed $3$-manifold, we use the rho
invariant to observe that if a $4$-manifold~$W$ admits a balanced plumbed
$3$-manifold as its boundary, then its signature defect vanishes; see
Proposition~\ref{prop:PbFilling}. Classical references on plumbed $3$-manifolds
include~\cite{NeumannCalculus,Hirzebruch71}. See also \cite{BorodzikFriedlPowell}
for their use in our context.
\medbreak
We begin by setting up notation.
\begin{construction}\label{constr:PlumbedManifold}
Let $G = (V,E)$ be an unoriented graph with no loops. The set~$E$ is the set of oriented edges, and $s\colon E \to V$
and $t\colon E \to V$ are the source and the target maps. The involution~$i \colon E \to E$ sends an oriented edge to the corresponding edge with the opposite orientation; see e.g.~\cite[Section I.2]{Serre80}. The graph is unoriented in the sense that for each edge, the set $E$ also contains the edge with the opposite orientation. We shall sometimes also denote $i(e)$ by $\bar e$.
Assume that the set of vertices~$V$ consists of oriented, connected and compact surfaces~$F$
and that the edges~$e\in E$ are labeled by weights~$\varepsilon(e)=\varepsilon(\bar e)\in\{\pm 1\}$.

For each edge $e$, we choose an embedded disc~$D_e \subset s(e)$ in such a way that no two discs intersect. 
We then remove these discs, by defining for each surface $F \in V$ the complement
\[ F^\circ = F \sm \bigcup_{s(e) = F} D_e. \]
We define the \emph{plumbed $3$-manifold} $\Pb(G)$ as
\[ \Pb(G):=\left( \bigsqcup_{F \in V} F^\circ \times S^1\right) / \sim \]
where, for all $e \in E$ the identifications are given by
\begin{align}\label{eq:plumbingmaps}
(-\partial D_{e}) \times S^1 &\to (-\partial D_{i(e)}) \times S^1\\
(x,y) & \mapsto
\nonumber \begin{cases}
(y^{-1}, x^{-1}),       & \text{if } \varepsilon(e)=1,\\
(y,x), 	 		& \text{if } \varepsilon(e)=-1.
\end{cases}
\end{align}
Since these identifications make use of orientation reversing homeomorphisms,
the $3$-manifold~$\Pb(G)$ carries an orientation that extends the orientation
of each~$F^\circ \times S^1$.
\end{construction}
\begin{remark}
The orientation $-\partial D_e$ is the one obtained by considering the circle
as a boundary component of $F^\circ$. This
is the opposite of the one induced by the boundary~$\partial D_e$ of the removed disk.
In the general context of plumbing disk bundles, one trivializes over the removed disks, which
causes the two formulas to flip; see e.g.~\cite[Chapter 8 p.\ 67]{Hirzebruch71}.
\end{remark}

The boundary of a plumbed $3$-manifold~$\Pb(G)$ is a union of tori and
the components correspond to the boundary components of the surfaces~$F \in V$.
By construction, the boundary components come with the
product structure~$\partial \Pb(G) = \bigsqcup_{F \in V}\partial F \times S^1$.
We define the homology class~$[\partial F] = [\partial F \times \{ pt\}]$
in $H_1(\partial \Pb(G); \R)$.

In order to describe the kernel $H_1(\partial \Pb(G);\R)\to H_1(\Pb(G);\R)$, we introduce some more notation:
for each surface~$F \in V$ with boundary, label its
boundary components~$K_1, \dots, K_{n_F}$ and accordingly their
meridians~$\mu_1^F, \dots, \mu_{n_F}^F$ and longitudes~$l_1^F, \dots, l_{n_F}^F$.
We have the equality~$[\partial F] = \sum_{k = 1}^{n_F} [l_k^F]$.
The vertices of our graph~$G$ are surfaces. So, for each edge~$e \in E$, the expression~$t(e)$ denotes a surface and~$\mu^{t(e)}_i$ denotes the meridian of $i$--th boundary torus of $t(e)$.
The following lemma describes the kernel of the inclusion~$H_1(\partial \Pb(G);\R)\to H_1(\Pb(G);\R)$,
which will be useful for our applications of Novikov-Wall additivity.
\begin{lemma}\label{lem:KernelPb}
The kernel of the inclusion induced map~$H_1(\partial \Pb(G);\R)\to H_1(\Pb(G);\R)$
is freely generated by the elements
\[ [\partial F] - \sum_{s(e) = F} \varepsilon(e) \mu_1^{t(e)} \text{ and } \quad  \mu_i^F - \mu_1^F, \]
for $F$ varying over the elements in $V$ with $\partial F \ne \emptyset$ and $2\leq i\leq n_F$.
\end{lemma}
\begin{proof}
From the construction of $\Pb(G)$, we see that for every edge~$e \in E$
there is a torus~$-\partial D_e \times S^1\subset s(e)\times S^1$ which is
identified with $-\partial D_{\bar{e}} \times S^1\subset t(e) \times S^1$. We denote this
torus by $T_e \subset \Pb(G)$. Hence, $T_e=-T_{\bar{e}}$.

Now pick an orientation~$E' \subset E$
on the edges, i.e.\ for every $e \in E$, exactly
one of the edges~$e$ and $\bar e$ is an element of $E'$.
From the construction of $\Pb(G)$, we obtain a Mayer-Vietoris sequence
	\[ \label{eq:mayvet} \dots \to  \bigoplus_{e \in E'} H_1(T_e;\R)
	\xrightarrow{i_t-i_s} \bigoplus_{F \in V} H_1(F^\circ \times S^1;\R) \to H_1(\Pb(G);\R) \to \cdots, \]
	where $i_t, i_s$ denote the maps induced by the inclusions of $T_e$ into $t(e) \times S^1$ and $s(e)\times S^1$ respectively.
	For each $F$, the inclusion~$\partial F\times S^1\to \Pb(G)$ factors through
	the space~$\bigsqcup_{F \in V} F^\circ \times S^1$. Consequently, we have the commutative diagram
	of inclusion induced maps
	\[
	\xymatrix@C0.8 cm@R0.4cm{
	\displaystyle\bigoplus_{e \in E'} H_1(T_e;\R) \ar[r]^-{i_t-i_s} & \displaystyle\bigoplus_{F \in V} H_1(F^\circ \times S^1;\R)
	\ar[r]^-h& H_1(\Pb(G);\R)\\
	&&H_1(\partial\Pb(G);\R) \ar[ul]^f \ar[u]^j,&
}
	\]
	yielding $\ker j = \ker  h\circ f  = \{x\in H_1(\partial\Pb(G);\R) \, \vert \, f(x) \in \im i_t-i_s \}.$
We shall now restrict our attention to those surfaces~$F$ with $\partial F\ne \emptyset$,
and prove that both $\mu_k^F - \mu_1^F$ and $[\partial F] - \sum_{s(e) = F} \varepsilon(e) \mu_1^{t(e)}$
belong to $\ker j$.
As $F^\circ$ is connected, all elements $\mu_k^F$ for $1\leq k \leq n_F$
are equal in $H_1(\Pb(G);\R)$, so the elements $\mu_k^F - \mu_1^F$ are in $\ker f$ and a fortiori in $\ker j$.
Next, we check that an element of the form~$[\partial F] - \sum_{s(e) = F} \varepsilon(e) \mu_1^{t(e)}$
is sent by $f$ to the image of $i_s-i_t$.
Note that  $H_1(F^\circ \times S^1; \R) = H_1(F^\circ; \R) \oplus \R\langle \mu_1^F\rangle$, so that we have the relation
	$[\partial F] + \sum_{s(e) = F} [- \partial D_e] = 0$ in $H_1(F^\circ \times S^1; \R)$.
We thus obtain
\[f\Big([\partial F] - \sum_{s(e) = F} \varepsilon(e) \mu_1^{t(e)}\Big)
		= \sum_{s(e) = F} \left([\partial D_e] - \varepsilon(e) \mu_1^{t(e)}\right), \]
and the claim reduces to checking that this element is
in the image of $i_t - i_s$.
Consider the class~$-[\partial D_e] \in H_1(T_e; \Z)$.
We have $- i_s [-\partial D_e] = [\partial D_e]$ and, by the gluing map
given in Construction~\ref{constr:PlumbedManifold},
$i_t[-\partial D_e]= -\varepsilon(e) \mu_{t(e)}$.
As a result, the difference $[\partial D_e] - \varepsilon(e) \mu_1^{t(e)}$ is indeed in the image of $i_t-i_s$, and so $ [\partial F] - \sum_{s(e) = F} \varepsilon(e) \mu_1^{t(e)}$ is in $\ker j$.

Note that the elements in the statement of the lemma span a subspace~$U$, whose
dimension is the number of boundary components of $\Pb(G)$, i.e.\ it is half the
dimension of the space~$H_1(\partial \Pb(G); \R)$.
By the half lives, half dies principle~\cite[Lemma 8.15]{Lickorish97}, the kernel~$\ker j$ has
the same dimension as $U$ and so coincides with $U$.
\end{proof}

\begin{definition}
\label{def:Balanced}
Let $G=(V,E)$ a graph with a label function $\varepsilon\colon E\to \{\pm 1\}$.
For $v,w \in V$ denote by $E(v, w) = \{ e \in E \mid s(e) = v, t(e) = w\}$
the set of all edges between $v$ and $w$. We call the integer
$p(v,w):= \sum_{e\in E(v, w)} \varepsilon(e)$ the \emph{total weight} of the pair of distinct vertices $(v,w)$.
The graph $G$ is called \emph{balanced} if $p(v,w)=0$ for all such pairs~$(v,w)$.
\end{definition}

From now on, assume that our plumbed $3$-manifold~$\Pb(G)$ comes with a homomorphism~$\phi\colon H_1(\Pb(G);\Z) \to \mathbb{Z}^\mu$.  We call such a homomorphism \emph{meridional} if, for each
constituting piece $F^\circ \times S^1\subseteq \Pb(G)$ with $F \in V$,
the restriction of $\phi$ to $H_1(F^\circ \times S^1;\Z)$ sends the class of
$\{pt\}\times S^1$ to one of the canonical generators $e_1,\dotsc, e_\mu$ of $\Z^\mu$.
Moreover, in the next two results we will restrict our attention to plumbings
of \emph{closed} surfaces.

The next lemma shows that if $G$ is balanced, then $\Pb(G)$ is cobordant to a
disjoint union of trivial surface bundles, where the cobordism has vanishing
signature defect.
\begin{lemma}\label{lem:handles}
Let $G = (V,E)$ be a balanced graph with vertices closed connected surfaces.
Suppose that $\phi \colon H_1(\Pb(G); \Z) \to \Z^\mu$ is a meridional homomorphism.
Then there exists a smooth $4$-manifold $Z$ over $\Z^\mu$ such that:
	\begin{enumerate}
		\item the boundary of $Z$ is a disjoint union
		\[\partial Z = -\Pb(G) \sqcup  \bigsqcup_{F \in V} \Sigma_F \times S^1,\]
		where every $\Sigma_F$ is a closed oriented surface;
		\item the restriction $H_1(\bigsqcup_{F \in V} \Sigma_F \times S^1; \Z) \to \Z^\mu$ is meridional;
		\item $\dsign_\omega Z = 0$ for all $\omega \in \T^\mu$.
	\end{enumerate}
\end{lemma}

\begin{proof}
Instead of proving the statement directly, we prove the following: if $E$ is nonempty, then
there exists a balanced graph~$G' = (V', E')$
with the same number of vertices and fewer edges than $G$, such that there exists a manifold~$Z_{G'}$ over $\Z^\mu$ with
$\partial Z_{G'} = -\Pb(G) \sqcup  \Pb(G')$,
which induces a meridional homomorphism on $\Pb(G')$ and such that $\dsign_\omega Z_{G'}= 0$ for all $\omega \in \T^\mu$. 

The original statement can be recovered as follows: iterate the above
to obtain a sequence of graphs~$G = G_0, \dots, G_n$ such that the set of edges of $G_n$ is empty.
Consequently, $\Pb(G_n) =  \bigsqcup_{F \in V} \Sigma_F \times S^1$.
We then glue the $4$-manifolds together:
$Z := Z_{G_1} \cup \dots \cup Z_{G_n}$. We get $\partial Z = -\Pb(G) \sqcup \Pb(G_n)$ as required
and by Novikov additivity $\dsign_\omega Z = \sum_{i=1}^n \dsign_\omega Z_{G_i} = 0$.

Now we proceed with the proof of the modified statement.
Recall from Construction~\ref{constr:PlumbedManifold} that to each edge~$e$ corresponds the
embedded torus~$T_e = (-\partial D_e) \times S^1$. The complement of all of these tori is diffeomorphic to
$\bigsqcup_{F \in V} F^\circ \times S^1 \subset \Pb(G)$.
In order to produce the desired $4$-manifold~$Z$, our aim is to attach a~$D^2 \times T^2$
to the trivial bordism~$\Pb(G) \times I$.

Given two vertices $F_1, F_2 \in V$, we write~$E(F_1, F_2) = \{ e \in E \mid s(e) = F_1, t(e) = F_2\}$ as
in Definition~\ref{def:Balanced}. Pick two vertices $F_1,F_2 \in V$ such that $E(F_1,F_2)$ is nonempty.
As the graph is balanced, this implies we can also pick two edges~$e,e' \in E(F_1,F_2)$ such that
$\varepsilon(e) = 1$ and $\varepsilon(e') = -1$.
Now set~$X_{e,e'} := I\times I \times S^1 \times S^1$.
Consider the corresponding tori~$T_e = (-\partial D_e) \times S^1$ and
$T_{e'} = (-\partial D_{e'}) \times S^1$, with oriented neighborhoods~$I \times T_e$, $I \times T_{e'}$.
	We attach~$X_{e,e'}$ to $\Pb(G)\times \{1\}$ along its vertical boundaries through a homeomorphism $f$ given by the following formulas:
	\begin{align*}\label{eq:handlegluing}
  \{ 0\} \times I \times S^1 \times S^1 &\to I \times (-\partial D_e) \times S^1  &      \{ 1\} \times I \times S^1 \times S^1 &\to I \times (-\partial D_{e'}) \times S^1 \\
	 \nonumber (0,t, x,y) &\mapsto (t, x,y), &
(1,t, x,y) &\mapsto (t, x^{-1},y) .
	\end{align*}
	The induced orientations on $\{0,1\}\times I\times S^1\times S^1$ are such that the above map is orientation-reversing. As a consequence, the orientations of $\Pb(G)\times I$ and $X_{e,e'}$ extend to the resulting $4$-manifold
\[	Z:= X_{e,e'} \cup_f \Pb(G) \times I. \]
Let $a_1, a_2\in \Z^\mu$ the images of the meridians of $F_1$ and $F_2$ under the map $H_1(\Pb(G);\Z)\to \Z^\mu$. Recalling the construction of $\Pb(G)$ given in \eqref{eq:plumbingmaps}, we see that the induced maps to $\Z^\mu$ on $T_e$ and $T_{e'}$ are given by
		\begin{align*}
H_1(-\partial D_e\times S^1;\Z)&\to \Z^\mu &H_1(-\partial D_{e'}\times S^1;\Z)&\to \Z^\mu\\
[\{p\}\times S^1]&\mapsto a_1     &    [\{p\}\times S^1]&\mapsto a_1\\
[-\partial D_e\times \{p\}]&\mapsto a_2 &[-\partial D_{e'}\times \{p\}]&\mapsto -a_2.
\end{align*}
The difference in the sign of the image $[-\partial D_e\times \{p\}]$ is
a consequence of the fact that the edges~$e,e'$ had opposite signs. This allows
us to define a map~$\phi_X \colon H_1(X_{e,e'}; \Z) \to \Z^\mu$
which glues with the map $\phi \colon H_1(\Pb(G); \Z) \to \Z^\mu$, i.e.\
the following diagram commutes:
\[
\begin{tikzcd}[column sep=0.6cm,row sep=0.6cm]
H_1(\{ 0,1\} \times I \times S^1 \times S^1; \Z) \ar[rr, "f_*"]\ar[dr, "\phi_X"] & &H_1(I \times T_e;\Z)\oplus H_1(I \times T_e';\Z) \ar[dl, "\phi"]\\
&\Z^\mu.&
\end{tikzcd}
\]

\begin{figure}[ht]
\includegraphics[width=11cm]{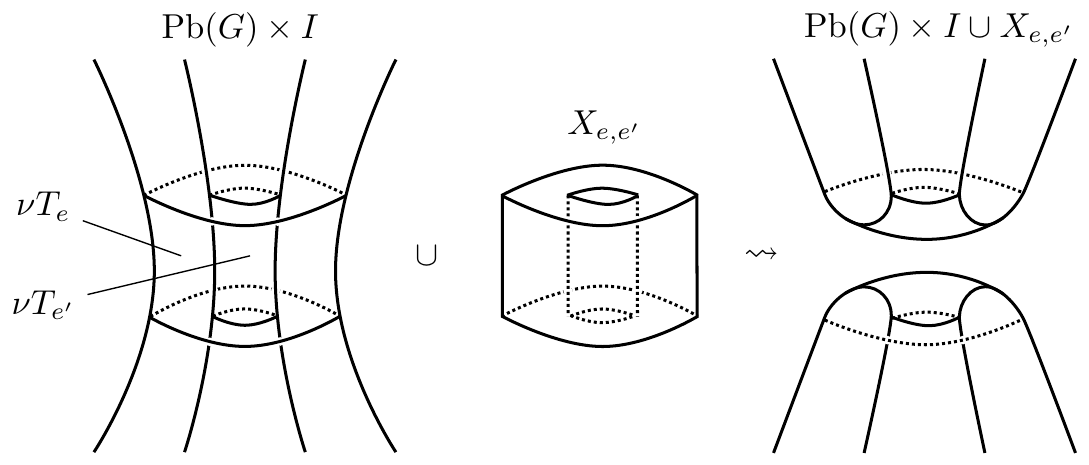}
\caption{The effect of attaching $X_{e,e'}$ to  $\Pb(G) \times I$ depicted in reduced dimensions}
\end{figure}
By making an additional choice of a splitting of the Mayer-Vietoris
sequence
\[ H_1(X_{e,e'};\Z) \oplus H_1(\Pb(G)\times I; \Z) \to H_1(Z; \Z) \to H_0(\{0,1\} \times I \times T^2;\Z), \]
we obtain a map $H_1(Z; \Z) \to \Z^\mu$ which extends
$\phi$ and $\phi_X$ on $H_1(\Pb(G)\times I; \Z)$ and $H_1(X_{e,e'};\Z)$.

The boundary of $Z$ has two components. The bottom boundary is $-\Pb(G)$.
The effect of adding $X_{e,e'}$ on the top boundary is
that of cutting along $T_e$ and $T_{e'}$ and gluing together the boundary component~$- \partial D_e \times S^1$
to $- \partial D_{e'} \times S^1$, and glueing $-\partial D_{i(e)} \times S^1$ to $-\partial D_{i(e')} \times S^1$.
Let $F_1' = F_1 \# T^2$ be the result of $0$--surgery along $D_e$ and $D_{e'}$
in $F_1$, and define $F_2'$ similarly.
The top boundary inherits a plumbed structure along a graph $G'$ obtained from $G$ by replacing the vertices $F_1$ and $F_2$ with $F_1'$ and $F_2'$, 
and by removing the edges~$e$ and $e'$.

We have verified that $Z$ fulfills the first statement.
To conclude the proof of the proposition, it remains to prove that~$\dsign_\omega Z = 0$. This
is a consequence of the following claim.
\begin{claim}
The twisted and untwisted signature of $\Pb(G) \times I$ and $X_{e,e'}$
vanish and Novikov-Wall additivity holds when gluing these two pieces together.
\end{claim}
To prove that the signatures vanish, note that both spaces are $4$-manifolds~$W$
with the property that the inclusions of the boundary $H_2(\partial W; \Z) \to H_2(W; \Z)$
and $H_2(\partial W; \C^\omega) \to H_2(W; \C^\omega)$ surject.
This implies that both the twisted and untwisted intersection forms vanish.
In particular, the twisted and untwisted signatures of $\Pb(G) \times I$ and $X_{e,e'}$ are zero.

Next, we consider Novikov-Wall additivity. We are gluing $W_+ = X_{e,e'}$ to
$W_- =\Pb(G)\times I$ along $M = \nu T_e \sqcup \nu T_{e'} \subset \Pb(G) \times \{1\}$. In the notations of Section \ref{sub:NovikovWall}, we have $N_+=I \times \{0,1\} \times S^1 \times S^1$ and $N_- = \Pb(G) \sm M$.
The boundary of the gluing region is given by the four tori
\[ \Sigma:= -\partial D_e \times S^1 \sqcup -\partial D_{i(e)} \times S^1 \sqcup -\partial D_{e'} \times S^1 \sqcup -\partial D_{i(e')} \times S^1 .\]
We shall prove that $V_{N_+} = \ker H_1(\Sigma;\R) \to H_1(N_+; \R)$
and $V_{N_-} = \ker H_1(\Sigma;\R) \to H_1(N_-; \R)$ agree, so that the hypotheses of the Novikov-Wall additivity theorem are satisfied (recall Theorem \ref{thm:Wall}) .

Observing the gluing maps above, we see that the vector space $V_{N_+}$ has basis
\begin{equation} \label{eq:KernelBasis}
[-\partial D_e] + [-\partial D_{e'}], \quad [ S^1_{e}] - [S^1_{e'}], \quad
	[-\partial D_{i(e)}] + [-\partial D_{i(e')}], \quad [ S^1_{i(e)}] - [S^1_{i(e')}].
\end{equation}
In order to describe $V_{N_-}$, observe that  $N_-=\Pb(G) \sm M$ inherits
a plumbed structure from $\Pb(G)$. It has the same surfaces as vertex set with $F_1$ and $F_2$
replaced by $F_1 \sm ( D_e \cup D_{e'} )$ and $F_2 \sm ( D_{i(e)} \cup D_{i(e')})$. Its set of
edges is obtained by removing $e$ and $e'$ from the set of edges of $G$.
Note that $\Sigma = \partial \Pb(G) \sm M$ and we can use
Lemma~\ref{lem:KernelPb} to obtain a basis for $V_{N_-}$. The difference of meridians gives the basis elements
$[ S^1_{e}] - [S^1_{e'}], [ S^1_{i(e)}] - [S^1_{i(e')}]$.
The surface~$F_1$ has boundary $-\partial D_e \sqcup -\partial D_{e'}$, so that further elements of the basis are given by
\[ [-\partial D_e] + [-\partial D_{e'}] - \sum_{s(k) = v} \varepsilon(k) \mu_k
	= [-\partial D_e] + [-\partial D_{e'}],\]
where the equality follows from the fact that $G$ is balanced.
The analogous statements holds for the other surface $F_2$. Consequently,
the vector space $V_{N_-}$ admits the same basis~\eqref{eq:KernelBasis} as $V_{N_+}$, and hence they coincide.
In particular, Theorem~\ref{thm:Wall} applies, and the untwisted signature
is additive.

For the twisted signature, thanks to Proposition~\ref{prop:TwistedWall},
it is enough to prove that the twisted homology vanishes for $\Sigma$.
This happens exactly if the induced $U(1)$-representation is nontrivial.
This is the case, because $\phi$ is meridional and the entries of $\omega$ are taken to be different from $1$. Consequently, the signature
defect is additive and so
\[ \dsign_\omega Z = \dsign_\omega \Pb(G) \times I + \dsign_\omega X_{e,e'} = 0.\]
\end{proof}

Using Lemma~\ref{lem:handles}, we can prove our main result about plumbed manifolds.
\begin{proposition}\label{prop:PbFilling}
	Let $G = (V,E)$ be a balanced graph with vertices closed connected surfaces~$F$.
Suppose that $\phi \colon H_1(\Pb(G); \Z) \to \Z^\mu$ is a meridional homomorphism and that
	 $\Pb(G)$ bounds a $4$-manifold~$W$ over $\Z^\mu$. Then, for all $\omega \in \T^\mu$,
	\[ \sign_\omega W -\sign W=0. \]
\end{proposition}
\begin{proof}
	Since the graph is balanced, Lemma~\ref{lem:handles} produces closed
surfaces $\Sigma_F$ and a $4$-manifold $Z$ over $\mathbb{Z}^\mu$ whose
signature defect vanishes, with boundary
\[\partial Z = -\Pb(G) \sqcup  \bigsqcup_{F \in V} \Sigma_F \times S^1.\]
One can now define $P:= W\cup_{\Pb(G)} Z$.  Since the boundary of $P$ consists
of a disjoint union of $\Sigma_F \times S^1$,
Corollary~\ref{cor:twSignatureVanishing} guaranties that $ \dsign_\omega P=0$.
As we are gluing along a full boundary component, Novikov additivity holds for both the twisted and untwisted signature, leading to $\dsign_\omega P = \dsign_\omega W + \dsign_\omega Z$.
Since we know that both $\dsign_\omega P$ and
$\dsign_\omega Z$ vanish, $\dsign_\omega W$ also vanishes.
\end{proof}

\subsection{Surfaces in the \texorpdfstring{$4$--ball}{4-ball}}
\label{sub:SurfacesD4}
In the remainder of the paper, plumbed $3$-manifolds will mostly appear as boundaries of tubular neighborhoods of collections of surfaces in the $4$-ball.
\medbreak

We observe that the exterior of a bounding surface contains a plumbed $3$-manifold in its boundary.
\begin{definition}\label{defn:IntersectionGraph}
The \emph{intersection graph~$(V,E)$ of a bounding surface}~$F = F_1\cup \dots \cup F_m$ has
the vertex set~$V = \{ F_1, \dots, F_m\}$. The set of edges~$E$ consists
of triples~$e = (x, F_i, F_j)$ where $x$ is an intersection point between
the components~$F_i, F_j \in V$. The maps $s,t,i$ are defined on $e$ by
\[ s(e) = F_i \quad t(e) = F_j \quad i(e) = (x, F_j, F_i) . \]
Moreover, we assign a weight $\varepsilon(e)=\pm 1$
to each edge~$e= (x, F_i, F_j)$ corresponding to the sign of the intersection at the point~$x$.
\end{definition}
Our interest in plumbed $3$-manifolds essentially lies in the next example,
which is only balanced if the link has 
pairwise vanishing
linking numbers.
\begin{example}\label{ex:PlumbingIntersections}
Let $F \subset D^4$ be a bounding surface for a link~$L$.
The boundary of the exterior~$W_F = D^4 \sm \nu F$ decomposes
into $\partial W_F = X_L \cup_{L\times S^1} M_F$.
Plumbing the trivialized disk bundles~$F_i \times D^2$ by the intersection graph of $F$ describes
a neighborhood~$\nu F$ of $F$. In this model, the surfaces~$F_i$ are recovered
as the zero sections~$F_i \times \{0\}$~\cite[Chapter 8]{Hirzebruch71}.
As consequence, we see that $M_F$ is diffeomorphic to $\Pb(G)$, where $G$ is the intersection graph of $F$.
\end{example}

Let~$F=F_1\cup\cdots \cup F_m$ be a bounding surface for a link~$L$, and let $L_i$ be the sublink given by $\partial F_i$, for $i=1,\dotsc , m$.
Denote as usual the exterior of $F$ by $W_F$.
Recall from Example~\ref{ex:PlumbingIntersections} that
$\partial W_F = X_L \cup_{L\times S^1} M_F$, where $M_F$ is a
plumbed $3$-manifold. Enumerate the components of $L_i$ and denote their meridians by
$\mu_k^{L_i}$ for $1 \leq k \leq n_{L_i}$, where $n_{L_i}$ is the number of components of $L_i$.
Define the linking number between two disjoint sublinks by
\[
\lk(L_i,L_j)=\sum_{\substack{K\subset L_i\\ J\subset L_j}}\lk(K,J) ,
\]
where the sum runs over the link components of~$L_i$ and $L_j$,
and set $\lk(L_i,L_i)=0$ for all $i$. The following computation will turn out to be useful when applying Novikov-Wall additivity.

\begin{lemma}\label{lem:SameKernel}
The vector space $V_{M_F} =  \ker H_1(L\times S^1; \R) \to H_1(M_F; \R) $ is generated by the elements of the form
\[    [L_i] - \sum_{j=1}^m \lk(L_i, L_j) \mu_{1}^{L_j}\quad \text{and} \quad  \mu_{k}^{L_i} - \mu_1^{L_i}.\]
\end{lemma}
\begin{proof}
Consider the surface~$F = t(e)$ for an edge~$e$ and the corresponding sublink~$\partial F = \partial t(e) \subset S^3$, whose first component has meridian~$\mu^{\partial t(e)}_1$.
Applying Lemma~\ref{lem:KernelPb}, the component $F_i$ gives rise to the basis vectors
\[ [L_i] - \sum_{s(e) = F_i} \varepsilon(e) \mu_{1}^{\partial t(e)} \quad \text{and} \quad  \mu_{k}^{L_i} - \mu_1^ {L_i}.\]
The result follows by observing that
\[    \sum_{s(e) = F_i} \varepsilon(e) \mu_{1}^{\partial t(e)}
= \sum_{j=1}^m (F_i\cdot F_j) \mu_{1}^{L_j}
= \sum_{j=1}^m \lk(L_i, L_j) \mu_{1}^{L_j}.\]
\end{proof}

\section{Invariance by \texorpdfstring{$0.5$--solvable}{0.5-solvable} cobordisms}\label{sec:Solvable}
The aim of this section is to prove that the multivariable signature and nullity are invariant under $0.5$-solvable cobordism. Sections~\ref{sub:H1Bordism} and~\ref{sub:05Solvable} respectively review the notion of  $H_1$-cobordisms and $0.5$-solvable cobordisms. Section~\ref{sub:NullityProof} tackles the invariance of the nullity. Section~\ref{sub:SignatureProof} is concerned with invariance of the signature. Finally, Section~\ref{sub:Technical} proves some technical results which are used in Sections~\ref{sub:NullityProof} and~\ref{sub:SignatureProof}

\subsection{\texorpdfstring{$H_1$--cobordisms}{H1-cobordism}}\label{sub:H1Bordism}
In this section, we review the definition of an $H_1$-cobordisms between 3-manifolds and prove some elementary properties following~\cite{Cha}.
\medbreak
A \emph{cobordism}~$(W; M, M', \varphi)$ between two connected $3$-manifolds~$M, M'$ with a preferred orientation-preserving diffeomorphism~$\varphi\colon \partial M \to \partial M'$ is a compact connected $4$-manifold~$W$ with a decomposition~$\partial W \cong -M \cup_{\varphi} M'$. We will often suppress~$\varphi$ from the notation.
A cobordism $(W; M, M')$ is an \emph{$H_1$-cobordism}
if additionally the inclusions of $M$ and $M'$ into $W$ induce
isomorphisms~$ H_1(M; \Z) \xrightarrow{\cong} H_1(W; \Z) \xleftarrow{\cong} H_1(M' ; \Z)$.

We start by recalling some immediate facts about $H_1$-cobordisms.
\begin{lemma}\label{lem:H1bordism}
If $(W;M,M')$ is an $H_1$-cobordism, then the following statements hold:
\begin{enumerate}
\item $H_i(W,M; \Z)=0=H_i(W,M';\Z)$ for all $i \neq 2$.
\item The groups $H_2(W,M;\Z)$ and $H_2(W,M';\Z)$ are isomorphic and free abelian.
\item Denote by $k \colon H_2(\partial W; \Z) \to H_2(W; \Z)$ the map induced by the inclusion.
There exists a unique map~$\psi \colon H_2(W,M; \Z) \to H_2(W; \Z)/\im k$ such that
\[
\begin{tikzcd}
H_2(W; \Z)/\im k \ar[r] & H_2(W, \partial W; \Z)\\
H_2(W; \Z) \ar[u] \ar[r]& H_2(W,M; \Z) \ar[u] \ar[ul, dashed, "\psi"]
\end{tikzcd}
\]
is commutative. The map~$\psi$ is an isomorphism.
\end{enumerate}
\end{lemma}
\begin{proof}
Since the first two assertions can be found in~\cite[Lemma 2.20]{Cha}, we only show here the third one here.
As a first step, we show that the map $i \colon H_2(W,M; \Z) \to H_2(W,\partial W; \Z)$
arising from the long exact sequence of the triple~$(W,\partial W,M)$ is an injection. To prove this,
consider the diagram
$$\xymatrix@C0.6cm@R0.8cm{
\text{Hom}(H_1(W; \Z),\mathbb{Z}) \ar[r] & \text{Hom}(H_1(M'; \Z),\mathbb{Z}) \\
H^1(W; \Z) \ar[r]^f \ar[dd]^{\text{PD}}_\cong \ar[u]^\cong_{\text{ev}} & H^1(M'; \Z) \ar[d]^{\text{PD}}_\cong \ar[u]^\cong_{\text{ev}}\\
 & H_2(M',\partial M'; \Z) \ar[d]^{\text{exc}}_\cong  \\
H_3(W,\partial W; \Z) \ar[r] & H_2(\partial W, M; \Z) \ar[r] & H_2(W,M; \Z) \ar[r]^i & H_2(W,\partial W; \Z),
}$$
where $\text{exc}$ denotes excision. The upper square clearly commutes, while the pentagon commutes
by~\cite[Section $\text{VI}.6$, Problem $3$]{Bredon}.
Since $(W;M,M')$ is an $H_1$-cobordism, the uppermost horizontal map is an isomorphism. Consequently,
the map $f$ is an isomorphism and therefore so is the map~$H_3(W,\partial W; \Z) \to H_2(\partial W,M; \Z)$.
Exactness now implies that $i \colon H_2(W,M; \Z) \to H_2(W,\partial W; \Z)$ is injective.

As a second step, we show existence and uniqueness of~$\psi \colon H_2(W,M; \Z) \to \frac{H_2(W; \Z)}{\im k}$.
The portion \[ H_2(\partial W; \Z) \stackrel{k}{\to} H_2(W; \Z) \stackrel{j}{\to} H_2(W,\partial W; \Z) \stackrel{\partial}{\to} H_1(\partial W; \Z) \stackrel{\ell}{\to} H_1(W; \Z) \]
of the long exact sequence of the pair $(W,\partial W)$ produces the short exact sequence in the top row of the following commutative diagram:
\[\xymatrix{
0 \ar[r] & \frac{H_2(W; \Z)}{\im k} \ar[r]^-{j} & H_2(W,\partial W; \Z) \ar[r]^\partial & \ker \ell  \ar[r]& 0 \\
& H_2(W; \Z) \ar[r] \ar@{->>}[u]& H_2(W,M; \Z) \ar[u]^i \ar@{-->}[lu]^{\psi} \ar[r] & \ker \big( H_1(M; \Z) \to H_1(W; \Z) \big). \ar[u]
}\]
Since $(W;M,M')$ is an $H_1$-cobordism, the group $\ker (H_1(M; \Z) \to
H_1(W; \Z))$ vanishes. Consequently, given $x \in H_2(W,M; \Z)$, the composition
$\partial(i(x))$ is zero and so, by exactness, there exists $[y] \in
\frac{H_2(W; \Z)}{\im k}$ such that $j([y])=i(x)$. We therefore define
$\psi(x):=[y]$. As $j$ is injective, $\psi$ is well-defined.
By construction $j \circ \psi=i$.

Next, we show that $\psi$ is an isomorphism. Injectivity is
immediate from the diagram above and the fact that $i$ is injective.
As $\ker (H_1(M; \Z) \to H_1(W; \Z))  = 0$, we obtain the following commutative diagram
\[\begin{tikzcd}
H_2(W; \Z)/\im k \ar[r] & H_2(W, \partial W; \Z)\\
H_2(W; \Z) \ar[u, twoheadrightarrow] \ar[r, twoheadrightarrow]& H_2(W,M; \Z) \ar[u] \ar[ul, "\psi"]
\end{tikzcd},\]
which shows the surjectivity of $\psi$.
\end{proof}

Given an $H_1$-cobordism $(W;M,M')$ with a map $H_1(W;\Z) \to \Z^\mu$, we shall often consider homology and cohomology with twisted coefficients in either $R=\Q(\Z^\mu)$ or $R=\C^\omega$ (for $\omega\in \T_!^\mu$). In both cases, we denote the underlying fields~$\Q(\Z^\mu)$ or
$\C$ by $\F$, so that the twisted (co-)homology groups are vector spaces over $\F$.
As in Section~\ref{sub:Twisted}, for a pair $(X,Y)$ we denote by
$\beta_i(X,Y)$ the rank of $H_i(X,Y;\Z)$ and by $\beta_i^\omega(X,Y)$ the dimension of $H_i(X,Y;\mathbb{C}^\omega)$.
We conclude this subsection with a consequence of Lemma~\ref{lem:Cone}.

\begin{lemma}\label{lem:ChainHomotopy}
Let $(W;M,M')$ be an $H_1$-cobordism
equipped with a homomorphism $H_1(W;\Z) \to \Z^\mu$.
Then both $H_i(W,M;\Q(\Z^\mu))$ and $H_i(W,M;\C^\omega)$ vanish for $i \neq 2$
and for all $\omega \in \T_!^\mu$. In particular, $\beta_2^\omega(W,M)$ equals $\beta_2(W,M)$.
\end{lemma}
\begin{proof}
Let $R = \Q(\Z^\mu)$ or $\C^\omega$.
Since $W$ is an $H_1$-cobordism, Lemma~\ref{lem:H1bordism} ensures that~$H_i(W,M;\Z)=0$ for $i \neq 2$. Lemma~\ref{lem:Cone} implies that $H_i(W,M;R)=0$ for $i=0,1$ and Lemma~\ref{lem:DualityUCSS} guarantees that for $i=3,4$, we have
\[ H_i(W, M; R) \cong H^{4-i} (W, M'; R) \cong \operatorname{Hom}_\F( H_{4-i}(W, M'; R), \F )^{\operatorname{tr}}
= 0.\]
The last claim now follows since the Euler characteristic of $(W,M)$ may be computed indifferently using $\Z$-coefficients or $R$-coefficients.
\end{proof}

\subsection{\texorpdfstring{$0.5$--solvable}{0.5-solvable} cobordisms}
\label{sub:05Solvable}
We review here the notion of $0.5$-solvable cobordism as defined in~\cite{Cha}. For simplicity, we avoid discussing $n$-solvability and $n.5$-solvability, referring to~\cite{Cha} for a more general treatment.
\medbreak
In the following paragraphs, given an $H_1$-cobordism $(W;M,M')$, we use $H$ as a shorthand for $H_1(W;\Z)$ and use $\lambda_1$ to denote the $\Z[H]$-valued intersection form on $H_2(W;\Z[H])$.
Recall the following definition from~\cite[Definition 2.8]{Cha},
which extends the definition of solvability from Cochran-Orr-Teichner's work~\cite{CochranOrrTeichner} to a relative notion.
\begin{definition}\label{def:nSolvable}
An $H_1$-cobordism~$(W;M,M',\varphi)$ is a \emph{$0.5$-solvable cobordism} if there exists a submodule~$\mathcal{L} = \langle l_1, \ldots, l_r \rangle \subset H_2(W;\Z[H])$ together with homology classes $d_1,\ldots, d_r \in H_2(W;\Z)$ that satisfy the following properties:
\begin{enumerate}
\item the intersection form~$\lambda_1$ vanishes on $\mathcal{L}$;
\item the image of $\mathcal{L}$ under the composition
$H_2(W;\Z[H]) \to H_2(W;\Z) \to H_2(W,M;\Z)$ has rank~$r \geq \frac{1}{2} \rk H_2(W, M; \Z)$;
\item the images~$l_i' \in H_2(W; \Z)$ of the elements~$l_i$ fulfill the relation~$\lambda_1(l_i',d_j)=\delta_{ij}$ for each~$1 \leq i,j \leq  r$;
\end{enumerate}
We refer to $\mathcal{L}$ as a \emph{$1$--lagrangian}, and to $\mathcal{D}:=\langle d_1,\dots, d_r\rangle$ as its \emph{$0$--dual}.
\end{definition}

\begin{remark}
Suppose that $(W; M, M')$ is a $0.5$--solvable cobordism $(W; M, M')$ with $1$--langrangian~$\mathcal{L}= \langle l_1, \ldots, l_r \rangle$. Then the images of the $l_i$'s in $H_2(W,M;\Z)$ span a free submodule of rank~$r$, since they are dual to the~$d_i$'s.
\end{remark}

For further reference, we make note of the following result, whose proof is outlined in~\cite[Proof of Theorem $3.2$]{Cha}.
\begin{proposition}\label{prop:ordinarysignature}
The signature~$\sign W$ of a $0.5$--solvable cobordism~$(W; M, M')$ vanishes.
\end{proposition}
\begin{proof}
Let $\mathcal{L}_{M,\Z}$ be the image of the $1$--lagrangian~$\mathcal{L}$ under~$H_2(W; \Z[H]) \to H_2(W, M; \Z)$. Let $\varphi \colon H_2(W,M; \Z) \to H_2(W;\Z) / \im  H_2(\partial W;\Z)$ be the isomorphism of Lemma~\ref{lem:H1bordism}.
The subspace $\varphi(\mathcal{L}_{M,\Z})$
is Lagrangian for the non-singular intersection pairing $\lambda_\Q$ of $W$.
Consequently, the signature of $W$ vanishes.
\end{proof}

The next definition is an adaptation to the colored framework of the definition given by Cha~\cite{Cha}. Recall that the boundary~$L \times S^1 = \partial X_L$ of a link exterior $X_L$
inherits a product structure by longitudes and meridians, which is well-defined up to isotopy.
A bijection~$\sigma$ of the link components of two links~$L, L'$
induces an orientation-preserving
diffeomorphism~$\varphi_\sigma \colon L \times S^1 \to L' \times S^1$
preserving the product structures, which is unique up to isotopy.
\begin{definition} \label{def:nSolvLink}
Two colored links~$L, L'$ are \emph{$0.5$-solvable cobordant}
if there exists a bijection~$\sigma$ between the components of $L$ and of $L'$
which preserves the colors and
there is a $0.5$-solvable cobordism $(W; X_L, X_{L'}, \varphi_\sigma)$.
\end{definition}
\begin{example} Suppose $L$ and $L'$ are concordant, and let $W$ be a concordance exterior. Then $(W; X_L, X_{L'})$ is a homology cobordism, which is a $0.5$--solvable cobordism since $H_2(W,X_{L};\Z)=0$.
\end{example}

Recall from Section~\ref{sub:Setup} that the exterior~$X_{L}$ of a $\mu$-colored link~$L$ is equipped
with a homomorphism~$\beta_{L} \colon H_1(X_{L}; \Z) \to \Z^\mu$. A $0.5$-solvable cobordism between two colored links $L$ and $L'$ fits into the commutative diagram
\begin{equation} \label{eq:CompLabel}
	\begin{tikzcd}
		H_1(X_{L}; \Z) \ar[r,"i"] \ar[rd, swap, "\beta_{L}"]  \ar[rr, bend left, "j_\sigma"]& H_1(W;\Z) & \ar[l,swap,"i'"]  \ar[ld, "\beta_{L'}"] H_1(X_{L'}; \Z) \\
		& \Z^\mu &
	\end{tikzcd}
\end{equation}
where $j_\sigma$ is the isomorphism that sends the meridian of a component $K$ of $L$ to the meridian of the corresponding component $\sigma(K)$ of $L'$.
We recall that the linking number between two disjoint sublinks is defined as the sum over the linking numbers of all their respective components (see Section \ref{sub:SurfacesD4}).

\begin{lemma}\label{rem:SameLN}
	Let $L$ and $L'$ be two $H_1$--cobordant oriented links. If~$(W; X_L, X_{L'},\varphi_\sigma)$ is a cobordism between them, then
	\[ \lk(J, K) = \lk(\sigma(J), \sigma(K)) \]
for each pair of components $J,K$ of $L$.
	In particular, if $L$ and $L'$ are concordant as $\mu$-colored links, then~$\lk(L_i, L_j) = \lk(L_i', L_j')$ for each pair of colors $i,j$.
\end{lemma}
\begin{proof}
The abelian group $H_1(X_{L}; \Z)$ is freely generated by the meridians of $L$, so that every element $x\in H_1(X_{L};\Z)$ has a well defined coordinate $x_K$ corresponding to the meridian of $K$. By definition, the linking number $\lk(J, K)$ is the coordinate~$b_K$ of the longitude $b$ of $J$.
		Let $b'\in H_1(X_{L'};\Z)$ be the longitude of $\sigma(J)$.
		Since the longitudes are glued together,
		we have $i(b)=i'(b') \in H_1(W; \Z)$ in Diagram~\eqref{eq:CompLabel}, and hence $j_\sigma(b)=b'$ by commutativity of the diagram.
		As the map $j_\sigma$ sends meridians to meridians, it preserves the coordinates, and hence $b'_{\sigma(K)}=b_K$. The proof of the first statement is concluded by observing that $b'_{\sigma(K)}$ is by definition the linking number between $\sigma(J)$ and $\sigma(K)$. The equality concerning $\mu$-colored links follows immediately from the fact that the cobordism preserves the colors.
\end{proof}

Given a $H_1$-cobordism $(W;M,M')$ with a map $H_1(W;\Z)\to \Z^\mu$,  the homomorphism $\Z[H_1(W;\Z)] \to \C$
and the canonical map~$\Z[H_1(W;\Z)] \to \Q(\Z^\mu)$ induce homomorphisms $i_R  \colon H_2(W;\Z[H_1(W;\Z)]) \to H_2(W;R)$ and $i_{M,R} \colon H_2(W;\Z[H_1(W;\Z)]) \to H_2(W,M;R)$,
 where $R$ stands either for~$\C^\omega$ or for $\Q(\Z^\mu)$. Also, we write $\lambda_R$ for the $\F$-valued intersection form on~$H_2(W;R)$.

The invariance of the signature and nullity will hinge on the following two results whose proof we delay until Section~\ref{sub:Technical}.

\begin{proposition}\label{prop:ComplexLagrangian}
Let $R$ be either $\Q(\Z^{\mu})$ or $\C^\omega$, with $\omega \in \T^\mu_!$. Let $(W; M, M')$ be an $0.5$--solvable cobordism over~$\Z^\mu$ with $1$--lagrangian~$\mathcal{L} = \langle l_1, \ldots, l_r \rangle$.
Then both subspaces
\begin{align*}
\mathcal{L}_{R} &= \langle i_{R} (l_1), \ldots, i_{R} (l_r)\rangle \subset H_2(W; R)\\
	\mathcal{L}_{M, R} &= \langle i_{M,R} (l_1), \ldots, i_{M,R} (l_r)\rangle \subset H_2(W, M; R)
\end{align*}
have dimension $r$.
Furthermore, they satisfy the following two properties:
\begin{enumerate}
\item the intersection form~$\lambda_{R}$ vanishes on $\mathcal{L}_{R}$.
\item $\dim \mathcal{L}_{R} = r  \geq \frac{1}{2}\dim_\Q H_2(W,M;\Q)$.
\end{enumerate}
\end{proposition}
\begin{proof} See Proposition~\ref{prop:free}. \end{proof}

When $R=\C^\omega$, we shall often drop the $\omega$ from the notation of the Lagrangian and simply write $\mathcal{L}_\C$.  The next proposition provides a lower bound on the dimension of~$\mathcal{L}_\C$.

\begin{proposition}\label{prop:LagrangianInequality}
Let $L, L'$ be two $\mu$-colored links that are $0.5$--solvable cobordant via~$(W;X_L, X_{L'})$ with $1$--lagrangian~$\mathcal{L}$. Then 
\[ \frac{1}{2} \dim_{\C} \left( \frac{H_2(W;\C^\omega)}{\im(H_2(\partial W;\C^\omega) \to H_2(W;\C^\omega))}\right) \leq \dim_{\C}(\mathcal{L}_\C).\]
\end{proposition}
\begin{proof} See Proposition~\ref{prop:DimensionCount}.\end{proof}
Using the two propositions above, we can now prove the invariance of the nullity and signature under 0.5-solvable cobordism.

\subsection{Nullities and \texorpdfstring{$0.5$--solvability}{0.5-solvability}}
\label{sub:NullityProof}
The next result states the invariance of the multivariable nullity and Alexander nullity under $0.5$-solvable cobordisms.

\begin{proposition}
\label{prop:Invariance}
	Let $R$ be either $\Q(\Z^{\mu})$ or $\C^\omega$, with $\omega \in \T_!^\mu$. If $(W;M,M')$ is a $0.5$--solvable cobordism, then the $\F$-vector spaces~$H_1(M;R)$ and $H_1(M'; R)$ have the same dimension. In particular, if $L$ and $L'$ are $0.5$-solvable cobordant links, then $\beta(L)=\beta(L')$ and $\eta_L(\omega)=\eta_L(\omega)$ for all $\omega \in\T^\mu_!$.
\end{proposition}
\begin{proof}
Consider the exact sequence of the pair~$(W, M)$ in which $R$ coefficients are understood:
\[ 0 \to \im \big( H_2(W) \stackrel{i_M}{\to} H_2(W, M) \big)  \to H_2(W,M)  \xrightarrow{\partial}  H_1(M) \to H_1(W) \to 0. \]
We use $\beta_i^R$ to denote the Betti numbers with $R$-coefficients. Since the Euler characteristic of the sequence is zero and since duality implies that $\beta_2^R(W,M)=\beta_2^R(W,M')$, the proposition boils down to showing that $\im(i_M)$ and $\im(i_{M'})$ have the same dimension. This is proved in Lemma~\ref{lem:RankImages} below.
\end{proof}

We are indebted to Christopher Davis for suggesting that we prove the following key lemma.
\begin{lemma}\label{lem:RankImages}
Let~$R$ be either $\Q(\Z^\mu)$ or $\C^\omega$, with $\omega \in \T_!^\mu$. The images of the two maps
\begin{align*}
   i_M \colon H_2(W; R) \to H_2(W, M; R)\\
   i_{M'}\colon H_2(W; R) \to H_2(W, M'; R).
\end{align*}
have the same dimension over $\F$.
\end{lemma}
\begin{proof}
Consider the three following intersection pairings:
\begin{align*}
\lambda_{W} \colon H_2(W; R) \times H_2(W; R) &\to \F,\\
\lambda_{W,\partial W} \colon H_2(W, \partial W; R) \times H_2(W; R) &\to \F,\\
\lambda_{W, M} \colon H_2(W, M; R) \times H_2(W, M'; R) &\to \F.
\end{align*}
These pairings are related as follows. First, observe that the map $i_{\partial W} \colon H_2(W;R) \to H_2(W,\partial W;R)$ induced by the inclusion factors as $i_{M,\partial W} \circ i_M$, where the map $i_{M,\partial W} \colon H_2(W, M;  R) \to H_2(W, \partial W;  R)$ is also induced by the inclusion. We introduce the same notation for $M'$, resulting in a map $i_{M',\partial W}$. Consider the following diagram:
\[ \xymatrix{
H_2(W; R) \ar[r]^{i_{\partial W}} \ar[rd]^{i_M}&  H_2(W,\partial W;R) \ar[r]^{\text{PD}} & H^2(W; R)  \ar[r]^-{\operatorname{ev}} & \operatorname{Hom}(H_2(W; R), \F)^t  \\
&  H_2(W,M; R) \ar[r]^{\text{PD}} \ar[u]_{i_{M,\partial W}} & H^2(W,M'; R) \ar[r]^-{\operatorname{ev}} \ar[u]^{i_{M'}^*}  & \operatorname{Hom}(H_2(W,M';R),\F)^t \ar[u]^{i_{M'}^*}. \\
} \]
The left triangle and right square are clearly commutative, while the middle square commutes thanks to~\cite[Section 6.9 Exercise 3]{Bredon}. It now follows that for $x,z$ in~$H_2(W;R)$ and $y$ in $H_2(W,M;R)$, we obtain
\begin{align}
\label{eq:Useful}
&   \lambda_{W,\partial W}(i_{\partial W}(x),z)=\lambda_W(x,z)=\lambda_{W,M}(i_M(x),i_{M'}(z)) \\
& \lambda_{W,\partial W}(i_{M,\partial W}(y),z)=\lambda_{W,M}(y,i_{M'}(z)).  \nonumber
\end{align}
We introduce one last piece of notation.
By Proposition~\ref{prop:ComplexLagrangian}, the subspaces~$\mathcal{L}_R$, $\mathcal{L}_M = i_M \big( \mathcal{L}_R \big)$, and $\mathcal{L}_{M'}= i_{M'} \big( \mathcal{L}_R \big)$ all have dimension~$r$. 

Now we construct a subspace
\[ \mathcal{D} = \langle d_1, \ldots, d_r \rangle \subset H_2(W, M; R) \]
by constructing the elements~$d_i$.
 Pick a basis~$\mathcal{L}_R = \langle l_1, \ldots, l_r \rangle$.
Since the dimension of $\mathcal{L}_{M'}= i_{M'} \big( \mathcal{L}_R \big)$ is $r$, the elements~$i_{M'}(l_j)$ form a basis of $\mathcal{L}_{M'}$. Therefore, the assignment~$i_{M'}(l_j) \mapsto \delta_{ij}$ defines a map~$\delta_i \colon \mathcal{L}_{M'} \to \F$. Since~$\F$ is a field, $\mathcal{L}_{M'} \subset H_2(W, M'; R)$ is a direct summand and consequently $\delta_i$ extends to an element~$\delta_i \in  \Hom_\F(H_2(W, M'; R), \F)^t$. The element~$d_i \in H_2(W, M; R)$ corresponds to $\delta_i$ under the isomorphism $H_2(W, M;R) \xrightarrow{\cong} \Hom_\F(H_2(W, M';R), \F)^t$ given by the adjoint of $\lambda_{W,M}$. This is an isomorphism, since the pairing~$\lambda_{W,M}$ is non-singular.

Consequently, the space $\mathcal{D}$ is freely generated by elements $d_1,\ldots, d_r$ that satisfy
\begin{equation}
\label{eq:Duals}
\lambda_{W, M}\big(d_i,  i_{M'} (l_j) \big) = \delta_{ij}.
\end{equation}
Completely analogously, we can define a subspace~$\mathcal{D}'$ of $H_2(W,M';R)$ with a basis given by $d_1',\ldots, d_r'$. Summarizing, we now have subspaces $\mathcal{L}_M$ and $\mathcal{D}$ of $H_2(W,M;R)$ and subspaces $\mathcal{L}_{M'}$ and $\mathcal{D}'$ of $H_2(W,M';R)$.

\begin{claim}
The subspaces $\mathcal{L}_M$ and $\mathcal{D}$ intersect trivially.
\end{claim}
To prove this, start with $a \in \mathcal{L}_M \cap \mathcal{D}$ and an arbitrary $l'$ in $\mathcal{L}_{M'}$. There is an $l$ in $\mathcal{L}_R$ such that $i_{M'}(l)=l'$. Similarly, since $a$ lies in $\mathcal{L}_M$, there is a $b$ in $\mathcal{L}_R$ such that $i_M(b)=a$.  Using~(\ref{eq:Useful}), we now have
\begin{equation}
\label{eq:UsefulTwo}
\lambda_{W,M}(a,l') = \lambda_{W,M}(i_M(b),i_{M'}(l))=
\lambda_W(b,l) = 0,
\end{equation}
where the last equality is due to the fact that $\mathcal{L}_R \subset \mathcal{L}_R^\perp$.
Since $a$ also lies in $\mathcal{D}$, we can write $a=\sum_i c_i d_i$. Combine Equation~(\ref{eq:UsefulTwo}) with the property of the $d_i$'s in Equation~(\ref{eq:Duals}) to deduce that~$0=\lambda_{W,M}\big(a,i_{M'}(l_j)\big)=c_j$ for each $j$. This implies that $a=0$, concluding the proof of the claim.

Using the claim it now makes sense to consider the direct sum $\mathcal{L}_M \oplus \mathcal{D} \subset H_2(W, M; R)$. Since $\mathcal{L}_M$ and $\mathcal{D}$ both have dimension at least $r$, we conclude that the dimension of $\mathcal{L}_M \oplus \mathcal{D}$ must at least be $2r$. Using Lemma~\ref{lem:ChainHomotopy}, we see that the dimension of $H_2(W, M; R)$ is equal to~$\rk H_2(W, M; \Z) \leq 2r$. Combining these observations and repeating them for~$\mathcal{D}'$, we deduce that~
\begin{align}
\label{eq:UsefulThree}
\mathcal{L}_M \oplus \mathcal{D} &= H_2(W, M; R), \\
\mathcal{L}_{M'} \oplus \mathcal{D}' &= H_2(W, M'; R), \nonumber\\
\label{eq:RankLM}
\dim \mathcal{L}_M &= r \text{ and } \dim \mathcal{L}_{M'}=r.
\end{align}
Recall that $i_M$ and $i_{M'}$ denote respectively the maps from $H_2(W; R)$ to $H_2(W, M; R)$ and $H_2(W, M'; R)$. Since, by definition, the subspaces~$\mathcal{L}_M$ and $\mathcal{L}_{M'}$ are images of~$\mathcal{L}_R$ under $i_M$ and $i_{M'}$, we deduce that they are subspaces of $\im(i_M)$ and $\im(i_{M'})$.

By~(\ref{eq:UsefulThree}), $\im(i_M)= \im(i_M) \cap (\mathcal{L}_M\oplus \mathcal{D})$, and the same for $M'$. Since we just argued that $\mathcal{L}_M \subset \im(i_M)$, and $\mathcal{L}_{M'} \subset \im(i_{M'})$, it follows that
\begin{align*}
&\dim \im(i_M) = \dim \mathcal{L}_M + \dim \big( \mathcal{D} \cap \im(i_{M}) \big) \\
&\dim \im(i_{M'}) = \dim \mathcal{L}_{M'} + \dim \big( \mathcal{D} \cap \im(i_{M'}) \big).
\end{align*}
Since we wish to show that $\dim \im(i_M)=\dim \im(i_{M'})$ and since $\mathcal{L}_M$ and $\mathcal{L}_{M'}$ have dimension~$r$ by Equation~(\ref{eq:RankLM}), it only remains to prove the following claim:
\begin{claim}
$\dim \big( \mathcal{D} \cap \im(i_M) \big) = \dim \big( \mathcal{D}' \cap \im(i_{M'}) \big)$.
\end{claim}
Since $\mathcal{D}$ and $\mathcal{D}'$ are freely generated by the $d_i$ and $d_i'$, there is an isomorphism $\psi \colon \mathcal{D} \to \mathcal{D}'$ obtained by mapping the $d_i$ to the $d_i'$. The claim will follow if we show that $\psi$ restricts to an isomorphism from $\mathcal{D} \cap \im(i_M)$ to $ {\mathcal{D}'} \cap \im(i_{M'})$. 

First, we check that the map~$\Psi$ restricts to a map~$\Psi|_{\mathcal{D} \cap \im(i_M)}  \colon \mathcal{D} \cap \im(i_M) \to \mathcal{D'} \cap \im(i_{M'})$. So assume that $x=\sum_i a_i d_i$ lies in $\im(i_M) \cap \mathcal{D}$. By definition $\psi(x)$ is equal to $x':=\sum_i a_i d_i'$, which clearly lies in $\mathcal{D}'$. Consequently we have to show that $x'$ lies in $\im(i_{M'})$. 
Since $x$ lies in $\im(i_{M})$, there is a $w$ in $H_2(W; R)$ such that $i_M(w)=x$. Now consider the element $v = x' - i_{M'}(w)$ of $H_2(W, M'; R)$: to show that $x'$ lies in $\im(i_{M'})$, it is enough to show that $v$ lies in $\im(i_{M'})$. Consequently, we consider the submodule
\[ \mathcal{L}^\perp_{M'} = \{ v \in H_2(W, M'; R) \colon \lambda_{W, M'}(v, i_M(l)) = 0 \text{ for all }   l\in \mathcal{L}_R\} \]
and start by verifying that $v \in \mathcal{L}^\perp_{M'}$. Recall that the~$l_j$'s form a basis of $\mathcal{L}_R$ and so it is enough to show that $\lambda_{W,M'}(v,i_{M}(l_j))$ vanishes for each~$j$. This follows successively by using the definition of $v$, the definition of the $d_i$'s in~\eqref{eq:Duals}, and the property in~(\ref{eq:Useful}):
\begin{align*}
\lambda_{W, M'}\big(v, i_M (l_j)\big)
&=\lambda_{W, M'}\big(x', i_M (l_j) \big) - \lambda_{W, M'} \big(i_{M'} (w), i_M (l_j) \big)\\
&=a_j - \lambda_W \big(w, l_j)\\
&=a_j - \lambda_{W, M} \big( x, i_{M'} (l_j) \big)\\
&= a_j - a_j =0.
\end{align*}
Note that $\mathcal{L}^\perp_{M'} \subset \mathcal{L}_{M'}$, since~$H_2(W, M'; R) = \mathcal{L}_{M'} \oplus \mathcal{D}'$ by the first claim above. Consequently, the vector~$v$ belongs to $\mathcal{L}_{M'}$ and thus to $\im(i_{M'})$. Since $v$ was defined as $x'-i_{M'}(w)$, we deduce that $x'$ must also lie in $\im(i_{M'})$, as desired. We showed that~$\Psi$ restricts to a map~$\Psi|_{\mathcal{D} \cap \im(i_M)}  \colon \mathcal{D} \cap \im(i_M) \to \mathcal{D'} \cap \im(i_{M'})$.

Now, by interchanging the roles of~$\mathcal D$ and $\mathcal D'$ in the argument above, we learn that the inverse~$\Psi^{-1}$ restricts to a map~$\Psi^{-1}|_{\mathcal{D'} \cap \im(i_{M'})} \colon \mathcal{D'} \cap \im(i_{M'}) \to \mathcal{D} \cap \im(i_{M})$. This restriction is the inverse of~$\Psi|_{\mathcal{D} \cap \im(i_M)}$ and thus the latter is an isomorphism. This concludes the proof of the last claim and thus of the proposition.
\end{proof}

\subsection{Signatures and \texorpdfstring{$0.5$--solvability}{0.5-solvability}}\label{sub:SignatureProof}
We prove that $0.5$-solvable cobordant links have the same multivariable signatures, concluding the proof of Theorem~\ref{thm:SolvableNullitySignature} from the introduction.

\begin{theorem}
If two $\mu$-colored links $L$ and $L'$ are $0.5$-solvable cobordant, then, for all $\omega \in \mathbb{T}^\mu_!$, we have
\[ \sigma_{L}(\omega)=\sigma_{L'}(\omega).\]
\end{theorem}

\begin{proof}
Let $F,F' \subset D^4$ be colored bounding surfaces for $L$ and $L'$ respectively,
with the additional requirement that they have only a single component
per color.
We denote by $W_{F}$ and $W_{F'}$ their respective exteriors and by $X,X'$ the link exteriors.
Setting as usual $M_{F} :=\partial\overline{\nu F}$, we see that the boundary~$\partial W_{F}$
decomposes into $X \cup_{L \times S^1} M_{F}$. An analogous decomposition holds for $\partial W_{F'}$. Let $W$ be a 0.5-solvable cobordism, with $\partial W= -X\cup_\varphi X'$, where $\varphi$ identifies $L\times S^1$ with $L'\times S^1$.
We consider the $4$-manifold
\[ V:=W_{F} \cup_{X} W \cup_{X'} (-W_{F'}),\]
which has boundary $M_{F} \cup_\Sigma (-M_{F'})$, where $\Sigma$ is a disjoint union of tori.
\begin{figure}[ht]
	\includegraphics[width=7cm]{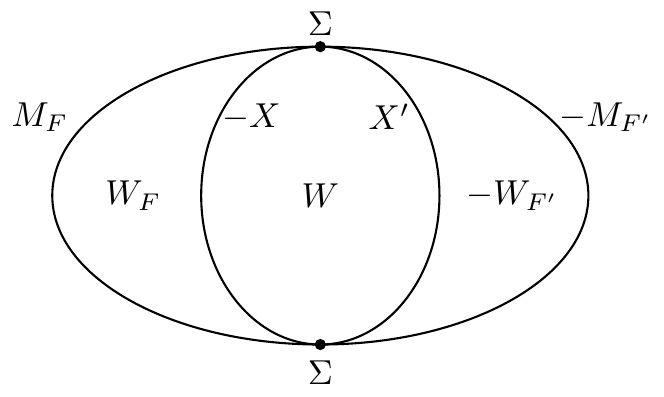}
	\caption{The manifold $V$ as a union of $W_{F}$, $W$ and $-W_{F'}$.}
	\label{fig:Sign0.5}
\end{figure}
By diagram~\eqref{eq:CompLabel}, the coefficient systems on the link exteriors
$X$ and $X'$ extend over $W$ and thus over $V$. We shall now compute
$\dsign_\omega(V) = \sign_\omega V - \sign V$ in two different ways.
\begin{claim}
$\dsign_\omega(V)=\dsign_\omega (W_{F}) - \dsign_\omega (W_{F'}) + \dsign_\omega (W).$
\end{claim}
The claim is proved by a double application of Novikov-Wall additivity, each time both for the twisted and untwisted signature: first we prove additivity for the gluing along $X$ of the two manifolds $W_F$ and $W\cup_{X'} (-W_{F'})$, and then for the gluing along $X'$ of $W$ with $-W_{F'}$. In both cases the boundary of the gluing region is $\Sigma=L\times S^1$, which is identified with $L'\times S^1$ through $\varphi$. As $H_1(\Sigma;\mathbb{C}^\omega)=0$, the hypotheses of Proposition \ref{prop:TwistedWall} are satisfied in the two cases, and twisted additivity holds. Let
\[\begin{split}
V_X&= \ker (H_1(\Sigma;\R) \to H_1(X;\R)), \quad V_{X'}= \ker (H_1(\Sigma;\R) \to H_1(X';\R)),\\
V_{M_{F}}&=\ker (H_1(\Sigma;\R) \to H_1(M_{F};\R)), \quad
	V_{M_{F'}}=\ker (H_1(\Sigma;\R) \to H_1(M_{F'};\R)).
	\end{split}
	\]
In the gluing along $X$, the three spaces to be considered for checking the hypotheses of Theorem \ref{thm:Wall} are $V_{M_{F}}$, $V_{M_{F'}}$, and $V_X$ in the same order as in the statement.
In the second gluing, it is $V_X$, $V_{M_{F'}}$, and $V_{X'}$. We show now that $V_{M_{F}}=V_{M_{F'}}$ and $V_X=V_{X'}$, so that the hypotheses are satisfied in both cases and additivity for the untwisted signature also holds. The space $V_{M_F}$ is described by Lemma~\ref{lem:SameKernel}. The space $V_{M_{F'}}$ is also described by Lemma~\ref{lem:SameKernel} as a subspace of $H_1(L'\times S^1;\R)$. By Lemma~\ref{rem:SameLN}, the two links have the same pairwise linking numbers. Since we assumed that $F$ and $F'$ have exactly one component for each color, the two vector spaces are seen to coincide under the identification between $L\times S^1 $ and $L'\times S^1$. The spaces $V_X$ and $V_{X'}$ also only depend on the linking numbers, and once again they coincide thanks to Lemma~\ref{rem:SameLN}. Hence, Novikov-Wall additivity holds both for the twisted and untwisted signature, and the claim is verified.

Thanks to Proposition~\ref{prop:UntwistedSign}, we have $\dsign_\omega (W_{F})=\sigma_{L}(\omega)$ and $\dsign_\omega (W_{F'})=\sigma_{L'}(\omega)$. The claim gets hence rewritten as
\[\dsign_\omega(V)=\sigma_{L}(\omega) - \sigma_{L'}(\omega) + \dsign_\omega (W).\]
We will now show that both signature defects $\dsign_\omega(W)$ and $\dsign_\omega(V)$ are actually~$0$, from which the conclusion follows.

By Proposition~\ref{prop:ordinarysignature}, the ordinary signature of $W$
vanishes. Invoking Proposition~\ref{prop:LagrangianInequality}, there exists a Lagrangian $\mathcal{L}_\mathbb{C} \subset H_2(W;\mathbb{C}^\omega)$
for the nonsingular intersection form on
$H_2(W;\mathbb{C}^\omega)/\im (H_2(\partial W;\mathbb{C}^\omega) \to H_2(W;\mathbb{C}^\omega) )$
and thus the twisted signature of $W$ must also vanish, so that $\dsign_\omega (W)=0$.

To conclude the proof, it only remains to show that $\dsign_\omega(V)=0$.
Recall that~$\partial V=M_{F} \cup_{\Sigma} (-M_{F'})$, where $\Sigma$ is a disjoint union of tori. We have seen in Example 4.12 that $M_{F}$ can be described as a plumbing of the components of $F$ along its intersection graph. In particular, the total weight between two vertices is given by
	\[p(F_i,F_j)=F_i\cdot F_j = \lk(L_i,L_j). \]
Similarly, the manifold $-M_{F'}$ is obtained by plumbing the surfaces $-F_1',\dotsc, -F_\mu'$, along the negative of the intersection graph of $F'$ (i.e.\ with its labels reversed), so that
	\[p(-F_i',-F_j')=-F_i'\cdot F_j' = -\lk(L_i',L_j'). \]
 The cobordism $W$ gives a bijection between the components of $L$ and those of $L'$, that induces homeomorphisms along which we can glue the components of $F$ and $F'$ in order to get closed oriented surfaces $G_i=F_i\cup_\partial -F'_i$ ($i=1,\dotsc, \mu$). Then $\partial V$ can be described as a plumbed $3$-manifold, whose plumbing graph has the surfaces $G_i$'s as vertices, and edges $E(G_i,G_j)=E(F_i,F_j)\sqcup E(-F_i',-F_j')$. In particular, for each pair of vertices, we have
\[p(G_i, G_j)=p(F_i,F_j)+p(-F_i',-F_j')= \lk(L_i, L_j) -\lk(L_i',L_j')=0, \]
as the linking numbers of $L$ and $L'$ match up.
This means that the plumbed $3$-manifold $\partial V$ is balanced, and Proposition~\ref{prop:PbFilling} now implies that $\dsign_\omega (V)=0$ as desired.
\end{proof}

\subsection{The proof of Proposition~\ref{prop:ComplexLagrangian} and Proposition~\ref{prop:LagrangianInequality}}
\label{sub:Technical}

At this stage, we have proved Theorem~\ref{thm:SolvableNullitySignature} skipping the proofs of Proposition~\ref{prop:ComplexLagrangian} and Proposition~\ref{prop:LagrangianInequality}. The aim of this last subsection is to prove these technical results, starting with some preliminary lemmas.
\medbreak

We consider the following set-up: let~$(W;M,M')$ be an $H_1$--bordism over~$\Z^\mu$, that is the cobordism is equipped with a commutative diagram
\[ \begin{tikzcd}
H_1(M; \Z) \ar[r, "\sim"] \ar[dr] & H_1(W; \Z) \ar[d]&\ar{l}[swap]{\sim} \ar[dl] H_1(M'; \Z)\\
&\Z^\mu&
\end{tikzcd}.
\]
We abbreviate~$H_1(W; \Z)$ by~$H$.
The composition~$\alpha \colon \Z[H] \to \Z[\Z^\mu] \to \C^\omega$ and the canonical inclusion $\Z[\Z^\mu] \to \Q(\Z^\mu)$ induce homomorphisms
\[ i_R  \colon H_2(W;\Z[H]) \to H_2(W;R) \text{ and } i_{M,R} \colon H_2(W;\Z[H]) \to H_2(W,M;R),\]
where $R$ stands for $\C^\omega$ or $\Q(\Z^\mu)$.

We start with a proposition whose proof is inspired by an argument of Cochran-Orr-Teichner ~\cite[Proposition 4.3]{CochranOrrTeichner}.
\begin{proposition}\label{prop:relIndep}
Let $R$ be either~$\Q(\Z^\mu)$ or~$\C^\omega$, with $\omega \in \T^\mu_!$.
Let $(W; M, M')$ be an $H_1$--cobordism over~$\Z^\mu$.
	Let~$\alpha_1, \ldots, \alpha_k \in H_2(W; \Z[H])$ be elements whose projections~$i_{M,\Z}(\alpha_i),\dotsc, i_{M,\Z}(\alpha_k) \in H_2(W, M; \Z)$ are linearly independent. Then, the elements~$i_{M,R} (\alpha_1),\dotsc, i_{M,R} (\alpha_k)$ are linearly independent in $H_2(W, M;R)$.
\end{proposition}
\begin{proof}
First, we establish suitable CW-structures on the manifolds~$W$ and $M$.
\begin{claim}
The pair~$(W,M)$ admits a finite CW-structure (up to homotopy), that is there
exists a finite CW-complexes~$W^c$ and a subcomplex $M^c \subset W^c$ with a diagram
\[ \begin{tikzcd}
W^c \ar[r, "\sim"] & W\\
M^c \ar[r, "\sim"] \ar[u,"\subset"] & M \ar[u, "\subset"].
\end{tikzcd}, \]
where the horizontal maps are homotopy equivalences and the diagram commutes up to homotopy.
Furthermore, we can pick~$M^c$ to be a $2$--dimensional complex and $W^c$ to be $3$--dimensional.
\end{claim}
Note that $M$ is a $3$--manifold with nonempty boundary, so it admits a smooth structure and one can find a $2$--dimensional CW-structure~$M^c \xrightarrow{\sim} M$ from a Morse function without critical points of index~$3$; see~\cite[Theorem 8.1 (Index 0)]{Milnor65}.

Since $W$ is a $4$--manifold with boundary, Poincaré duality shows that the homology group~$H_k(W; \Z[\pi_1(W)])$ vanishes for $k \geq 4$ (this involves an explicit computation of~$H_0(W, \partial W; \Z[\pi_1(W)])$).
The $4$--manifold~$W$ admits a finite CW-structure, since it is an absolute neighbourhood retract~\cite[Theorem 3.3]{Han51}; see \cite{Wes77}.
Using a result of Wall~\cite[Corollary 5.1]{Wall66}, these two facts imply that there exists a $3$--dimensional CW-structure~$W^c \xrightarrow{\sim} W$.

Use the inverse of the homotopy equivalence~$W^c \xrightarrow{\sim} W$ to
obtain a map~$M^c \to W^c$ and arrange $M^c$ to be a subcomplex by replacing~$W^c$ with the mapping cylinder of $M^c \to W^c$. Since~$M^c$ was a $2$--complex, $W^c$ is still $3$--dimensional.

Before we proceed with the next claim, note that the commutativity up to homotopy is exactly the ingredient needed to construct a map between the cylinders of the inclusions~$\Cyl(M^c \subset W^c) \xrightarrow{\sim} \Cyl(M \subset W)$. Consequently, the relative homology groups of $(W, M)$, and $(W^c, M^c)$ agree.
\begin{claim}
Without increasing the dimensions of $(W^c, M^c)$, we may assume that
there exists a subcomplex~$X \subset W^c$ disjoint from $M^c$ with 
\[ H_2(X; \Z[H]) = \Z[H]\langle \alpha_1, \ldots, \alpha_k \rangle.\]
\end{claim}
First, we realize the homology classes~$\alpha_i$ geometrically:
For each class~$\alpha_i \in H_2(W; \Z[H])$, there exists a closed oriented surface~$\Sigma_i$ together with a map~$f_i \colon \Sigma_i \to \widehat W$ such that $f_i([\Sigma_i]) = \alpha_i$~\cite[Théorème III.3]{Thom54}, where $\widehat W$ is the abelian cover of~$W$ corresponding to the composition $\pi_1(W)\to H_1(W;\Z)$.
Use the inverse of the homotopy~$W^c \to W$ to obtain maps~$f_i^c \colon \Sigma_i \to \widehat W^c$. 
Consider the space~$\widehat X = H \times \vee_i \Sigma_i$ on which $H$ acts by multiplication on the first factor. Define the $H$--equivariant map
\begin{align*}
  f\colon H \times \vee_i \Sigma_i &\to \widehat W^c\\
  \Big( h, x\Big) &\mapsto h \cdot f^c_i(x) \text{ for } x \in \Sigma_i.
\end{align*}
We now think of $\widehat{X}$ as a subspace of the mapping cylinder~$\Cyl(f)$.
Note that the quotient~$\Cyl(f)/ H \simeq W^c$, and $X = \widehat X / H$ is a subcomplex, which is also disjoint from~$M^c$. Replace~$W^c$ by  $\Cyl(f)/ H$, which is a $3$--dimensional complex, since~the surfaces~$\Sigma_i$ are $2$--dimensional.
The subcomplex~$X$ has homology~$H_2(X; \Z[H]) = \Z[H]\langle \alpha_1, \ldots, \alpha_k \rangle$, which is freely generated by the~$\alpha_i$.
This concludes the proof of the claim.

Having constructed suitable CW-structures on $W$ and $M$, we now proceed with the proof.  
Observe that the following quotient map is a chain isomorphism
\begin{equation}
\label{eq:Identify}
 C(M^c \sqcup X, M^c; Q) \cong C(X; Q),
 \end{equation}
where the coefficient system~$Q$ is either~$\Z$ or $R$. In particular, the assumption on the projections precisely means that the map~$H_2(M^c \sqcup X, M^c;\Z) \to H_2(W^c,M^c;\Z)$ is injective. Similarly, our goal is to show that the induced map~$H_2(M^c \sqcup X, M^c;R) \to H_2(W^c, M^c ;R)=H_2(W, M; R)$ is injective. Indeed, this map sends the $\F$--basis~$\{ \alpha_i \}$ of $H_2(M^c \sqcup X, M^c;R) \cong H_2(X;R)$ to the elements~$\{ i_{M,R}(\alpha_i) \}$.

In order to establish injectivity, consider the following exact sequence of the triple~$(M^c, M^c \sqcup X, W^c)$ with $Q = R$ coefficients:
\begin{equation}
\label{eq:LesTriple}
\begin{tikzcd}[column sep=2.8mm]
H_3(W^c, M^c; Q) \ar[r] & H_3(W^c, M^c \sqcup X; Q) \ar[r, "\partial^Q"] & H_2(M^c \sqcup X, M^c; Q) \ar[r, "i_Q"] & H_2(W^c, M^c; Q).
\end{tikzcd}
\end{equation}

Note that Lemma~\ref{lem:ChainHomotopy} shows that the homology group~$H_3(W,M;R) = H_3(W^c,M^c;R)$ vanishes. As we shall see below,
the proposition reduces to the following claim.

\begin{claim}
The homology group $H_3(W^c,M^c \sqcup X;R)$ vanishes.
\end{claim}
Consider the long exact sequence~\eqref{eq:LesTriple} above for $Q = \Z$.  Recall that $H_2(M^c \sqcup X, M;\Z) \to H_2(W^c, M^c;\Z) = H_2(W, M;\Z)$ is injective by assumption, and~$H_3(W^c, M^c;\Z) = H_3(W, M;\Z)$ vanishes since~$W$ is an $H_1$--bordism. This shows that~$H_3(W^c, M^c \sqcup X; \Z) = 0$. 
Since the CW-structure of $W$ has no $4$--cells, the (cellular) chain module
$C_4(W^c, M^c \sqcup X; \Z) = 0$. 
From these two facts, deduce that the boundary
operator~$\partial_3^\Z \colon C_3(W^c, M^c \sqcup X; \Z) \to C_2(W^c, M^c \sqcup X; \Z)$
is injective. Now we relate this observation to the case~$R = Q$:
$\partial_3 \colon  C_3(W^c, M^c \sqcup X; \Z[H]) \to C_2(W^c, M^c \sqcup X; \Z[H])$
is a homomorphism between free modules, and since $\partial_3^\Z \colon C_3(W^c, M^c \sqcup X; \Z) \to C_2(W^c, M^c \sqcup X; \Z)$ is injective, we deduce that~$\partial_3^{R} \colon C_3(W^c, M^c \sqcup X; R) \to C_2(W^c, M^c \sqcup X; R)$ is injective by Lemma~\ref{lem:DeterminantTrick}. This implies the claim that $H_3(W^c, M^c \sqcup X; R) = 0$.

We now conclude the proof of the proposition. Using the claim and~\eqref{eq:LesTriple}, we deduce that $i_R$ is injective. As we mentioned above, this shows that the $i_{M,R}(\alpha_i)$ are linearly independent and thus the proof is concluded.
\end{proof}

The next proposition was Proposition~\ref{prop:ComplexLagrangian} above.

\begin{proposition}\label{prop:free}
Let $\omega \in \T_!^\mu$ and let $R$ be either $\Q(\Z^\mu)$ or $\C^\omega$.
Let $(W; M, M')$ be an $0.5$--cobordism over~$\Z^\mu$ with $1$--langrangian~$\mathcal{L} = \langle l_1, \ldots, l_r \rangle$. Then the subspaces
\begin{align*}
\mathcal{L}_{R} &= \langle i_{R} (l_1), \ldots, i_{R} (l_r)\rangle \subset H_2(W; R)\\
\mathcal{L}_{M, R} &= \langle i_{M,R} (l_1), \ldots, i_{M,R} (l_r)\rangle \subset H_2(W, M; R)
\end{align*}
have dimension
$\dim \mathcal{L}_{R}=\dim \mathcal{L}_{M,R} = r  \geq \frac{1}{2}\dim_\Q H_2(W,M;\Q)$. Furthermore,
the intersection form~$\lambda_{R}$ vanishes on $\mathcal{L}_{R}$.
\end{proposition}
\begin{proof}
Denote the $0$--duals of~$W$ by~$\mathcal{D} = \langle d_1, \ldots, d_r\rangle$.
Denote~$H_1(W; \Z)$ by~$H$, and consider the map~$i_\Z \colon H_2(W; \Z[H]) \to H_2(W; \Z)$, which is induced by the augmentation map~$\Z[H] \to \Z$.
By definition of a $0.5$--cobordism, the images~$i_\Z (l_i) \in H_2(W; \Z)$ of the elements~$l_i$ fulfill the relation $\lambda_\Z \big(i_\Z (l_i),d_j)=\delta_{ij}$ for each~$1 \leq i,j \leq  r$.
This relation descends to the pairing
\[ \lambda_{M,\Z} \colon H_2(W, M; \Z) \times H_2(W, M'; \Z) \to \Z, \]
that sends $(i_{M,\Z} (l_i), d'_j)$ to $\lambda_{M,\Z}(i_{M,\Z} (l_i), d'_j) =\delta_{ij}$
 where $d'_j \in H_2(W, M';\Z)$ is the relative class of $d_j$, that is the image of $d_j$ under the map $H_2(W;\Z) \to H_2(W,M';\Z)$ induced by the canonical inclusion.
Consequently, the elements~$i_{M,\Z} (l_i)$ are linearly independent.
Now apply Proposition~\ref{prop:relIndep} to the elements~$l_i$ to see that the~$i_{M,R} (l_i)$'s are linearly independent.
Since $H_2(W; R) \to H_2(W, M; R)$ sends $i_{R} (l_i) \mapsto i_{M,R} (l_i)$, the elements~$i_{R} (l_i)$ are linearly independent as well.
\end{proof}

The final step is to prove Proposition~\ref{prop:LagrangianInequality}, that is the inequality
\[ \frac{1}{2} \dim_{\C} \left( \frac{H_2(W;\C^\omega)}{\im(H_2(\partial W;\C^\omega) \to H_2(W;\C^\omega))}\right) \leq \dim_{\C}(\mathcal{L}_\C)\]
for cobordisms between link exteriors. We start with two preliminary lemmas involving twisted Betti numbers.
\begin{lemma}\label{lem:Betti}
If two $\mu$-colored links $L$ and $L'$ are
$H_1$-cobordant via $(W; X_L, X_{L'})$,
then for $i = 1,2$ and for all $\omega \in \T^\mu_!$  we have
$$\beta_2^\omega(W,\partial W) = \beta_2^\omega(W,X_{L})+\beta_3^\omega(W,\partial W)= \beta_2^\omega(W, X_{L'})+\beta_3^\omega(W,\partial W).$$
\end{lemma}
\begin{proof}
We start by establishing two preliminary equalities.
As $X_L$ is a link exterior,
its Euler characteristic vanishes. Since $\beta_0^{\omega}(X_L)$ and
$\beta_3^\omega(X_L)$ vanish and since  $\chi^{\omega}(X_L)=\chi(X_L)=0$, we obtain
\begin{equation}
\label{eq:BettiNumberLinks}
\beta_1^\omega(X_{L})=\beta_2^\omega(X_{L}),
\end{equation}
and similarly for $L'$.
 Arguing as in Lemma~\ref{lem:splitboundary}, one deduces that $\beta_i^\omega(\partial W)=\beta_i^\omega(X_L)+\beta_i^\omega(X_{L'})$. Using~(\ref{eq:BettiNumberLinks}), we then see that $\beta_1^\omega(X_{L})-\beta_2^\omega(X_{L})$ equals $\beta_1^\omega(X_{L'})-\beta_2^\omega(X_{L'})$ and therefore
\begin{equation}
\label{eq:BettiNumberBoundaryW}
\beta_1^\omega(\partial W)=\beta_1^\omega(X_L)+\beta_1^\omega(X_{L'})=\beta_2^\omega(X_L)+\beta_2^\omega(X_{L'})=\beta_2^\omega(\partial W).
\end{equation}
We now prove the first equality displayed in the lemma (the proof of the second is identical). Lemma~\ref{lem:ChainHomotopy} shows
that both modules
$H_3(W,X_L;\mathbb{C}^\omega)=0$ and $H_1(W,X_L;\mathbb{C}^\omega)=0$ vanish.
Consider the long exact
sequence of the triple $(W,\partial W,X_L)$
\begin{align*}
0 \to H_3(W, \partial W;\C^\omega) &\to H_2(\partial W,X_L;\C^\omega) \to H_2(W,X_L;\C^\omega) \to H_2(W,\partial W;\C^\omega)\to\\
	&\to H_1(\partial W,X_L;\C^\omega) \to 0 \to H_1(W,\partial W;\C^\omega) \to 0
\end{align*}
and deduce that $H_1(W,\partial W;\mathbb{C}^\omega) =0$. Since the alternating sum of dimensions of an exact sequence vanishes, we obtain
\[ \beta_2^\omega(W,\partial W) =\beta_2^\omega(W,X_L)+\beta_3^\omega(W,\partial W)+\beta_1^\omega(\partial
W,X_L)-\beta_2^\omega(\partial W,X_L).\]
Thus the statement of the lemma reduces to proving the
equality~$\beta_1^\omega(\partial W,X_L)=\beta_2^\omega(\partial W,X_L)$.
To achieve this, consider the long exact sequence of the pair~$(\partial W,X_L)$:
\begin{align*}
 0 \to H_2(X_L;\C^\omega) &\to H_2(\partial W;\C^\omega) \to H_2(\partial W,X_L;\C^\omega) \\
 & \to H_1(X_L;\C^\omega) \to H_1(\partial W;\C^\omega)
\to H_1(\partial W,X_L;\C^\omega) \to 0.
\end{align*}
Note that $H_3(\partial W,X_L; \mathbb{C}^\omega)=0$ because of the long exact sequence of $(W,\partial W,X_L)$
together with the fact that $H_3(W,X_L; \mathbb{C}^\omega)=0$; see Lemma~\ref{lem:ChainHomotopy}.
Again, the alternate sum of dimensions
\[ \beta_2^\omega(X_L) - \beta_2^\omega (\partial W) + \beta_2^\omega (W, X_L) -\beta_1^\omega(X_L) + \beta_1^\omega(\partial W) - \beta_1^\omega(\partial W, X_L)= 0\]
vanishes, and the desired equality now follows by combining~(\ref{eq:BettiNumberLinks}) and~(\ref{eq:BettiNumberBoundaryW}).
\end{proof}

Next, we prove an inequality on the twisted Betti numbers of an
$H_1$-cobordism.
\begin{lemma} \label{lem:Inequality}
Let $L$ and $L'$ be links that are $H_1$--cobordant over~$\Z^\mu$ via $(W; X_L, X_{L'})$, then
\[\beta_3^\omega(W,\partial W)-\beta_1^\omega(\partial W)+\beta_1^\omega(W)\leq 0 \]
for all $\omega \in \mathbb{T}_!^\mu$.
\end{lemma}

\begin{proof}
By Lemma~\ref{lem:DualityUCSS}, one gets $\beta_3^\omega(W,\partial W)=\beta_1^\omega(W)$. Arguing as in the proof of Lemma~\ref{lem:Betti}, we see that $\beta_1^\omega(\partial W) =\beta_1^\omega(X_L)+\beta_1^\omega(X_{L'})$. Since Lemma~\ref{lem:ChainHomotopy} implies that $H_1(W,X_{L};\C^\omega)=0$, the map $H_1(X_{L};\C^\omega) \to H_1(W;\C^\omega)$ is surjective, and so $\beta_1^\omega (W)-\beta_1^\omega (X_{L}) \leq 0$ and similarly for $X_{L'}$. Combining these facts, $\beta_3^\omega(W,\partial W)-\beta_1^\omega(\partial W)+\beta_1^\omega(W)$ is equal to $2\beta_1^\omega(W)-\beta_1^\omega(X_{L})-\beta_1^\omega(X_{L'}) =(\beta_1^\omega(W)-\beta_1^\omega(X_{L'}))+(\beta_1^\omega(W)-\beta(X_{L'})) \leq 0$, as desired.
\end{proof}

\begin{proposition}\label{prop:DimensionCount}
Let $L$ and $L'$ be $\mu$-colored links that are $0.5$-solvable cobordant via $(W; X_L, X_{L'})$. Then, for all $\omega \in \T_!^\mu$, the subspace~$\mathcal{L}_\C \subset H_2(W;\C^\omega)$ of Proposition~\ref{prop:free} satisfies
\[ \frac{1}{2} \dim_{\C} \left( \frac{H_2(W;\C^\omega)}{\im(H_2(\partial W;\C^\omega) \to H_2(W;\C^\omega))}\right) \leq \dim_{\C}(\mathcal{L}_\C).\]
\end{proposition}
\begin{proof}
Invoking Proposition~\ref{prop:free}, the dimension of $\mathcal{L}_\C$ is larger than half the rank of $H_2(W,X_L;\Z)$. Using Lemma~\ref{lem:ChainHomotopy}, $\beta_2^\omega(W,X_L)=\beta_2(W,X_L)$, and so the proposition reduces to showing the inequality
\[ d:=\dim\left( \frac{H_2(W;\mathbb{C}^\omega)}{\im H_2(\partial W;\mathbb{C}^\omega) \to H_2(W;\C^\omega)}\right)
	\leq \beta_2^\omega(W, X_L).\]
Set $V:=\im ( H_2(\partial W;\mathbb{C}^\omega) \to H_2(W;\mathbb{C}^\omega) )$.
Since we proved in Lemma~\ref{lem:ChainHomotopy} that $H_1(W,\partial W;\mathbb{C}^\omega)$ vanishes,
the long exact sequence of the pair $(W,\partial W)$ now takes the form
\[ 0 \to V \to H_2(W;\mathbb{C}^\omega) \to H_2(W,\partial W;\mathbb{C}^\omega) \to H_1(\partial W;\mathbb{C}^\omega) \to H_1(W;\mathbb{C}^\omega) \to 0.\]
Finally, using the fact that the alternating dimensions of an exact sequence sum up to zero, one gets
\begin{align*}
d
 &=\beta_2^\omega(W,\partial W)-\beta_1^\omega(\partial W)+\beta_1^\omega(W)  \\
 &=\beta_2^\omega(W,X_L)+\beta_3^\omega(W,\partial W)-\beta_1^\omega(\partial W)+\beta_1^\omega(W)  \\
 & \leq \beta_2^\omega(W,X_L),
\end{align*}
where the last two steps use respectively Lemma~\ref{lem:Betti} and Lemma~\ref{lem:Inequality}.
\end{proof}

\bibliographystyle{alpha}
\bibliography{BiblioSolvable}

\begin{thebibliography}{HNK71}

\bibitem[APS75]{AtiyahPatodiSinger}
Michael~F. Atiyah, Vijay~K. Patodi, and Isador~M. Singer.
\newblock Spectral asymmetry and {R}iemannian geometry. {II}.
\newblock {\em Math. Proc. Cambridge Philos. Soc.}, 78(3):405--432, 1975.

\bibitem[BFP16]{BorodzikFriedlPowell}
Maciej Borodzik, Stefan Friedl, and Mark Powell.
\newblock Blanchfield forms and {G}ordian distance.
\newblock {\em J. Math. Soc. Japan}, 68(3):1047--1080, 2016.

\bibitem[BGV92]{Berline92}
Nicole Berline, Ezra Getzler, and Mich\`ele Vergne.
\newblock {\em Heat kernels and {D}irac operators}, volume 298 of {\em
  Grundlehren der Mathematischen Wissenschaften [Fundamental Principles of
  Mathematical Sciences]}.
\newblock Springer-Verlag, Berlin, 1992.

\bibitem[Bre93]{Bredon}
Glen~E. Bredon.
\newblock {\em Topology and geometry}, volume 139 of {\em Graduate Texts in
  Mathematics}.
\newblock Springer-Verlag, New York, 1993.

\bibitem[CCZ16]{CimasoniConwayZacharova}
David Cimasoni, Anthony Conway, and Kleopatra Zacharova.
\newblock Splitting numbers and signatures.
\newblock {\em Proc. Amer. Math. Soc.}, 144(12):5443--5455, 2016.

\bibitem[CF64]{ConnerFloyd}
Pierre~E. Conner and Edwin~E. Floyd.
\newblock {\em Differentiable periodic maps}.
\newblock Ergebnisse der Mathematik und ihrer Grenzgebiete, N. F., Band 33.
  Academic Press Inc., Publishers, New York; Springer-Verlag,
  Berlin-G\"ottingen-Heidelberg, 1964.

\bibitem[CF08]{CimasoniFlorens}
David Cimasoni and Vincent Florens.
\newblock Generalized {S}eifert surfaces and signatures of colored links.
\newblock {\em Trans. Amer. Math. Soc.}, 360(3):1223--1264 (electronic), 2008.

\bibitem[CFT16]{ConwayFriedlToffoli}
Anthony Conway, Stefan Friedl, and Enrico Toffoli.
\newblock The {B}lanchfield pairing of colored links.
\newblock {\em ArXiv e-prints}, 1609.08057, 2016.
\newblock \href{http://arxiv.org/abs/1609.08057}{ArXiv 1609.08057}.

\bibitem[Cha08]{Cha08}
Jae~Choon Cha.
\newblock Topological minimal genus and {$L^2$}-signatures.
\newblock {\em Algebr. Geom. Topol.}, 8(2):885--909, 2008.

\bibitem[Cha14]{Cha}
Jae~Choon Cha.
\newblock Symmetric {W}hitney tower cobordism for bordered 3-manifolds and
  links.
\newblock {\em Trans. Amer. Math. Soc.}, 366(6):3241--3273, 2014.

\bibitem[Cim04]{CimasoniPotential}
David Cimasoni.
\newblock A geometric construction of the {C}onway potential function.
\newblock {\em Comment. Math. Helv.}, 79(1):124--146, 2004.

\bibitem[Con17]{ConwayThesis}
Anthony Conway.
\newblock {\em Invariants of colored links and generalizations of the {B}urau
  representation}.
\newblock PhD thesis, Universit\'e de Gen\`eve, 2017.
\newblock
  \href{http://dx.doi.org/10.13097/archive-ouverte/unige:99352}{10.13097/archive-ouverte/unige:99352}.

\bibitem[Coo82]{Cooper}
Daryl Cooper.
\newblock The universal abelian cover of a link.
\newblock In {\em Low-dimensional topology ({B}angor, 1979)}, volume~48 of {\em
  London Math. Soc. Lecture Note Ser.}, pages 51--66. Cambridge Univ. Press,
  Cambridge-New York, 1982.

\bibitem[COT03]{CochranOrrTeichner}
Tim~D. Cochran, Kent~E. Orr, and Peter Teichner.
\newblock Knot concordance, {W}hitney towers and {$L^2$}-signatures.
\newblock {\em Ann. of Math. (2)}, 157(2):433--519, 2003.

\bibitem[DFL18]{DegtyarevFlorensLecuona2}
Alex Degtyarev, Vincent Florens, and Ana Lecuna.
\newblock Signature and slopes of colored links.
\newblock {\em ArXiv e-prints}, 1802.01836, 2018.
\newblock \href{http://arxiv.org/abs/1802.01836}{ArXiv 1802.01836}.

\bibitem[Han51]{Han51}
Olof Hanner.
\newblock Some theorems on absolute neighborhood retracts.
\newblock {\em Ark. Mat.}, 1:389--408, 1951.

\bibitem[Hil12]{Hillman}
Jonathan Hillman.
\newblock {\em Algebraic invariants of links}, volume~52 of {\em Series on
  Knots and Everything}.
\newblock World Scientific Publishing Co. Pte. Ltd., Hackensack, NJ, second
  edition, 2012.

\bibitem[HNK71]{Hirzebruch71}
Friedrich Hirzebruch, Walter~D. Neumann, and Sebastian~S. Koh.
\newblock {\em Differentiable manifolds and quadratic forms}.
\newblock Marcel Dekker, Inc., New York, 1971.
\newblock Appendix II by W. Scharlau, Lecture Notes in Pure and Applied
  Mathematics, Vol. 4.

\bibitem[HS97]{Hilton97}
Peter~J. Hilton and Urs Stammbach.
\newblock {\em A course in homological algebra}, volume~4 of {\em Graduate
  Texts in Mathematics}.
\newblock Springer-Verlag, New York, second edition, 1997.

\bibitem[Kim15]{Kim}
Min~Hoon Kim.
\newblock Whitney towers, gropes and {C}asson-{G}ordon style invariants of
  links.
\newblock {\em Algebr. Geom. Topol.}, 15(3):1813--1845, 2015.

\bibitem[Lev69]{Levine69}
Jerome Levine.
\newblock Knot cobordism groups in codimension two.
\newblock {\em Comment. Math. Helv.}, 44:229--244, 1969.

\bibitem[Lev77]{Levine77}
Jerome Levine.
\newblock Knot modules. {I}.
\newblock {\em Trans. Amer. Math. Soc.}, 229:1--50, 1977.

\bibitem[Lic97]{Lickorish97}
W.~B.~Raymond Lickorish.
\newblock {\em An introduction to knot theory}, volume 175 of {\em Graduate
  Texts in Mathematics}.
\newblock Springer-Verlag, New York, 1997.

\bibitem[Mil65]{Milnor65}
John Milnor.
\newblock {\em Lectures on the {$h$}-cobordism theorem}.
\newblock Notes by L. Siebenmann and J. Sondow. Princeton University Press,
  Princeton, N.J., 1965.

\bibitem[Mur65]{Murasugi}
Kunio Murasugi.
\newblock On a certain numerical invariant of link types.
\newblock {\em Trans. Amer. Math. Soc.}, 117:387--422, 1965.

\bibitem[Neu79]{Neumann79}
Walter~D. Neumann.
\newblock Signature related invariants of manifolds. {I}. {M}onodromy and
  {$\gamma $}-invariants.
\newblock {\em Topology}, 18(2):147--172, 1979.

\bibitem[Neu81]{NeumannCalculus}
Walter~D. Neumann.
\newblock A calculus for plumbing applied to the topology of complex surface
  singularities and degenerating complex curves.
\newblock {\em Trans. Amer. Math. Soc.}, 268(2):299--344, 1981.

\bibitem[NP17]{NagelPowell}
Matthias Nagel and Mark Powell.
\newblock Concordance invariance of {L}evine-{T}ristram signatures of links.
\newblock {\em Documenta Mathematica}, 22:25--43, 2017.

\bibitem[Pow17]{Powell}
Mark Powell.
\newblock The four-genus of a link, {L}evine--{T}ristram signatures and
  satellites.
\newblock {\em J. Knot Theory Ramifications}, 26(2):1740008, 28, 2017.

\bibitem[Ser80]{Serre80}
Jean-Pierre Serre.
\newblock {\em Trees}.
\newblock Springer-Verlag, Berlin-New York, 1980.
\newblock Translated from the French by John Stillwell.

\bibitem[Tan98]{Tanaka98}
Toshifumi Tanaka.
\newblock Four-genera of quasipositive knots.
\newblock {\em Topology Appl.}, 83(3):187--192, 1998.

\bibitem[Tho54]{Thom54}
Ren\'e Thom.
\newblock Quelques propri\'et\'es globales des vari\'et\'es diff\'erentiables.
\newblock {\em Comment. Math. Helv.}, 28:17--86, 1954.

\bibitem[Tri69]{Tristram}
Andrew~G. Tristram.
\newblock Some cobordism invariants for links.
\newblock {\em Proc. Cambridge Philos. Soc.}, 66:251--264, 1969.

\bibitem[Vir09]{Viro09}
Oleg Viro.
\newblock Twisted acyclicity of a circle and signatures of a link.
\newblock {\em J. Knot Theory Ramifications}, 18(6):729--755, 2009.

\bibitem[Wal66]{Wall66}
C.~Terence~C. Wall.
\newblock Finiteness conditions for {${\rm CW}$} complexes. {II}.
\newblock {\em Proc. Roy. Soc. Ser. A}, 295:129--139, 1966.

\bibitem[Wal69]{Wall}
C.~Terence~C. Wall.
\newblock Non-additivity of the signature.
\newblock {\em Invent. Math.}, 7:269--274, 1969.

\bibitem[Wes77]{Wes77}
James West.
\newblock Mapping {H}ilbert cube manifolds to {ANR}'s: a solution of a
  conjecture of {B}orsuk.
\newblock {\em Ann. of Math. (2)}, 106(1):1--18, 1977.

\end{thebibliography}
\end{document}